\documentclass{amsart}

\pdfoutput=1

\usepackage{amsmath,amsthm,amssymb} 
\usepackage{comment}

\usepackage{graphicx}

\theoremstyle{plain}
\newtheorem{theorem}{Theorem}[section]
\newtheorem{proposition}[theorem]{Proposition}
\newtheorem{lemma}[theorem]{Lemma}

\theoremstyle{definition}
\newtheorem{definition}[theorem]{Definition}
\newtheorem{example}[theorem]{Example}
\newtheorem{remark}[theorem]{Remark}

\usepackage{multicol}

\usepackage{wrapfig}

\usepackage{tabularx}
\usepackage{here}

\newcommand{\R}{{\mathbb R}}

\newcommand{\Z}{{\mathbb Z}}

\title
[Colorings, invariants and double coverings of twisted links]
{Twisted intersection colorings, invariants and double coverings of twisted links}

\author[H.~Ito]{Hiroki ITO}
\address[H.~Ito]{Department of Mathematics, Osaka University, 1-1 Machikaneyama, Toyonaka, Osaka 560-0043, Japan}
\email{hiroki.ito.math@gmail.com}

\author[S.~Kamada]{Seiichi KAMADA}
\address[S.~Kamada]{Department of Mathematics, Osaka University, 1-1 Machikaneyama, Toyonaka, Osaka 560-0043, Japan}
\email{kamada@math.sci.osaka-u.ac.jp}
\date{}


\begin{document}

\maketitle
%\thispagestyle{empty}

%=====================================================================================
\renewcommand{\abstractname}{Abstract}

\begin{abstract}
Twisted links are a generalization of classical links and correspond to stably equivalence classes of links in thickened surfaces. 
In this paper we introduce twisted intersection colorings of a diagram and construct two invariants of a twisted link using such colorings. 
As an application, we show that there exist infinitely many pairs of twisted links such that for each pair the two twisted links  are not equivalent but their double coverings are equivalent.   

We also introduce a method of constructing a pair of twisted links whose double coverings are equivalent.
\end{abstract}

% Table of Contents 
% \tableofcontents

%
%
\section{Introduction}

Twisted links are a generalization of classical links. 
They are represented by link diagrams possibly with virtual crossings and bars \cite{Bou}. 
As virtual links correspond to stably equivalence classes of links in thickened oriented surfaces \cite{CKS02, KK00, Kau99}, 
twisted links correspond to stably equivalence classes of links in thickened, possibly nonorientable, surfaces \cite{Bou}.  

One of the purpose of this paper is to introduce twisted intersection colorings of a diagram and to construct two invariants of a twisted link using such colorings.  We will construct three polynomials 
called index polynomials, $\psi$, $\psi_{+1}$ and $\psi_{-1}$, which are computed from a diagram with a twisted intersection coloring.  
The polynomial $\psi$ does not depend on the coloring, and it is an invariant of a twisted link.  
For the other two polynomials, $\psi_{+1}$ and $\psi_{-1}$, depend on the coloring. However the pair $(\psi_{+1}, \psi_{-1})$ under a certain equivalence relation, called $2$-equivalence, is an invariant of a twisted link.  This is our main result (Theorem~\ref{thm b}.)  

For a given twisted link, a virtual link called a {\it double covering} is defined in \cite{KK16}.  
It was unknown if there is a pair of distinct twisted links whose double coverings are equivalent as virtual links. 
As an application of Theorem~\ref{thm b} 
we show that there exist infinitely many pairs of twisted links such that for each pair the two twisted links are not equivalent but their double coverings are equivalent (Theorem~\ref{thm:doubleB}).  This is the second main result of this paper.  

Another purpose of this paper is to introduce a method of constructing a pair of twisted links such that their double coverings are the same.  The pairs used in Theorem~\ref{thm:doubleB} are constructed by this method.  

When we restrict our invariant $\psi$ to virtual links with one component, i.e., virtual knots,  the invariant $\psi$ is essentially the same with the affine index polynomial \cite{Kau12} or equivalently the writhe invariants in \cite{NNS} and \cite{ST}.  The restriction of $\psi$ to virtual links with two components is essentially the same with the linking polynomial invariant in defined \cite{CG}. On the other hand, the polynomials $\psi_{+1}$ and $\psi_{-1}$ are always trivial for virtual link diagrams.   

In \cite{K13} another polynomial invariant, also called the index polynomial, was defined for a twisted link.  
Our invariant is different from that in  \cite{K13}.  We show an example of a pair of twisted links such that  
the invariant in \cite{K13} does not distinguish them but our invariant does. 

The paper is organized as follows: In Section~\ref{sect:virtual-twisted} we recall the definitions of virtual links and twisted links. 
In Section~\ref{sect:twisted-coloring} we introduce the notion of a twisted intersection coloring of a twisted link diagram. 
In Section~\ref{sect:invariants} the three index polynomials, $\psi$, $\psi_{+1}$ and $\psi_{-1}$, and $2$-equivalence on pairs $(\psi_{+1}, \psi_{-1})$ are introduced.   Theorem~\ref{thm b} is proved there too.  
In Section~\ref{sect:double} we recall the notion of double coverings of twisted links and prove Theorem~\ref{thm:doubleB}.  In Section~\ref{sect:construction}  we introduce a method of constructing a pair of twisted links such that their double coverings are the same.

This research was supported by JSPS KAKENHI Grant Number JP19H01788.

\section{Virtual links and twisted links}
\label{sect:virtual-twisted}

A {\em virtual link diagram} is a link diagram in $\R^2$ possibly with some crossings, called virtual crossings, which are not assigned over/under information and are decorated with a small circle (Figure~\ref{fig:suri}).

\begin{figure}[htbp]
\begin{tabular}{cc}
\begin{minipage}{0.40\hsize}
\begin{center}
\includegraphics[height=1.8cm]{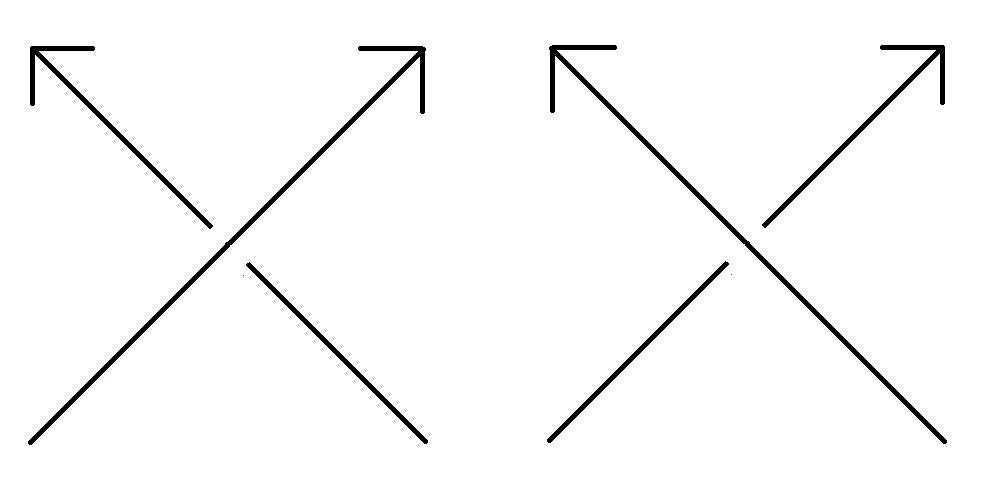}

A positive/negative crossing
\end{center}
\end{minipage}

\begin{minipage}{0.30\hsize}
\begin{center}
\includegraphics[height=1.8cm]{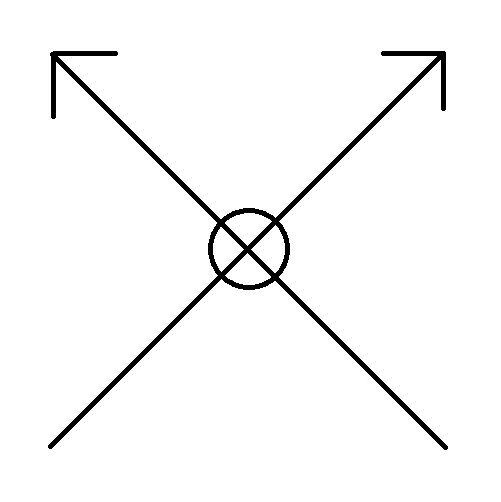}

A virtual crossing
\end{center}
\end{minipage}

\begin{minipage}{0.30\hsize}
\begin{center}
\includegraphics[height=1.8cm]{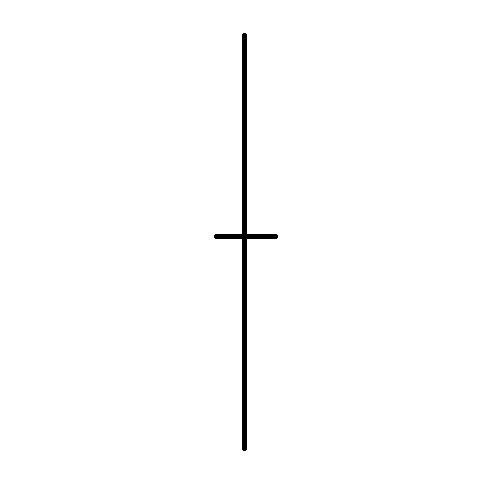}

A bar 
\end{center}
\end{minipage}
\end{tabular}
\caption{Classical crossings, a virtual crossing and a bar}
\label{fig:suri}
\end{figure}

\vspace{-0.3cm}

Local moves R1, R2, R3, V1, \dots, V4 depicted in Figure~\ref{fig:suri-R1-V4} are called 
{\em generalized Reidemeister moves}.  
Two diagrams $D$ and $D'$ are said to be {\em equivalent as virtual links} if they are related by a finite sequence of isotopies of $\R^2$ and generalized Reidemeister moves.  We denote it by $D\overset{\rm v}{\sim} D'$.  
A {\em virtual link} is an equivalence class of virtual link diagrams under the equivalence relation $\overset{\rm v}{\sim}$, \cite{Kau99}.

\begin{figure}[htbp]
\begin{tabular}{cc}
\begin{minipage}{0.30\hsize}
\begin{center}
\includegraphics[height=1.4cm]{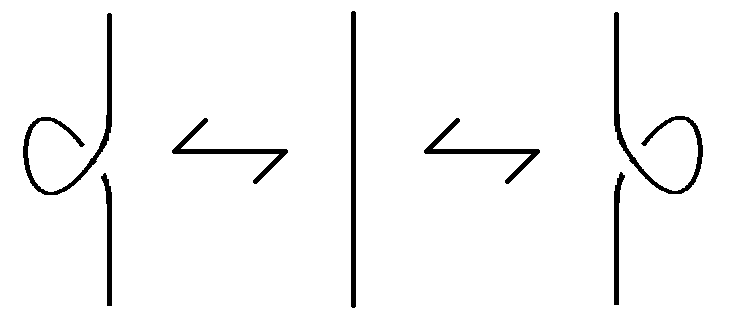}

${\rm R1}$
\end{center}
\end{minipage}
\begin{minipage}{0.30\hsize}
\begin{center}
\includegraphics[height=1.4cm]{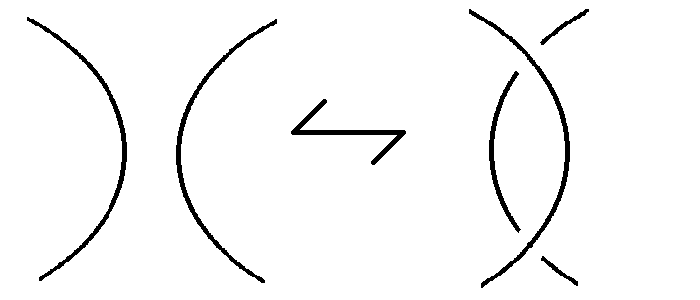}

${\rm R2}$
\end{center}
\end{minipage}
\begin{minipage}{0.30\hsize}
\begin{center}
\includegraphics[height=1.4cm]{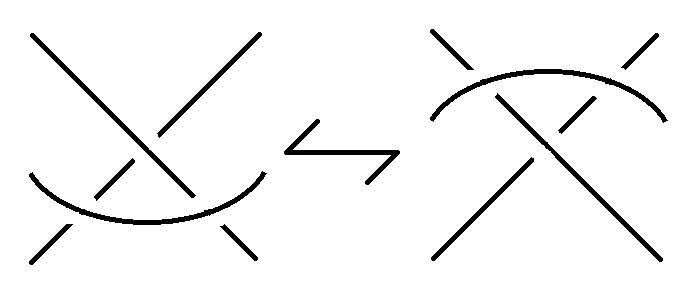}

${\rm R3}$
\end{center}
\end{minipage}
\end{tabular}
\end{figure}

\vspace{-0.5cm}

\begin{figure}[htbp]
\begin{tabular}{cc}
\begin{minipage}{0.20\hsize}
\begin{center}
\includegraphics[height=1.5cm]{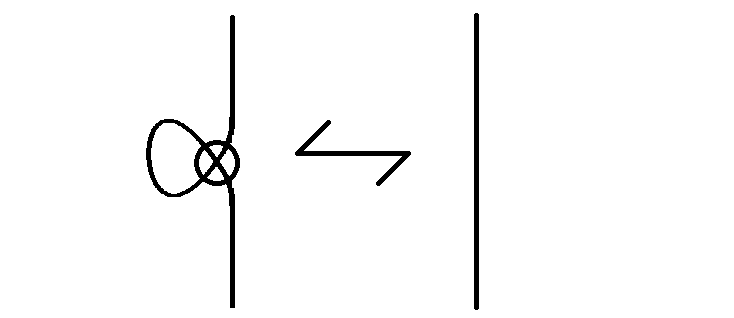}

${\rm V1}$
\end{center}
\end{minipage}
\begin{minipage}{0.25\hsize}
\begin{center}
\includegraphics[height=1.5cm]{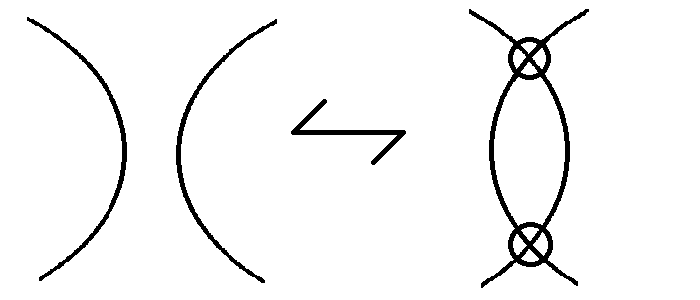}

${\rm V2}$
\end{center}
\end{minipage}
\begin{minipage}{0.25\hsize}
\begin{center}
\includegraphics[height=1.5cm]{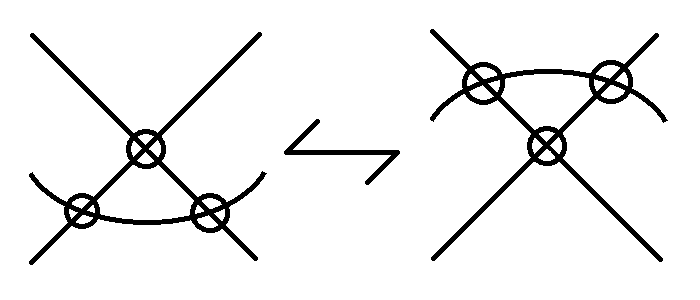}

${\rm V3}$
\end{center}
\end{minipage}
\begin{minipage}{0.25\hsize}
\begin{center}
\includegraphics[height=1.5cm]{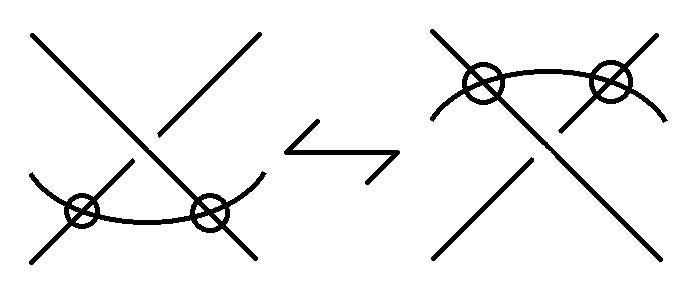}

${\rm V4}$
\end{center}
\end{minipage}
\end{tabular}
\caption{Generalized Reidemeister moves}
\label{fig:suri-R1-V4}
\end{figure}

\vspace{-0.3cm}

A {\em twisted link diagram} is a link diagram possibly with some virtual crossings and some bars, where a bar is a short segment intersecting an arc of the diagram as in Figure~\ref{fig:suri}.  
Local moves R1, R2, R3, V1, \dots, V4, T1, T2 and T3 (Figures~\ref{fig:suri-R1-V4} and~\ref{fig:suri-T1-T3}) 
are called {\em extended Reidemeister moves}.  
Two diagrams $D$ and $D'$ are said to be {\em equivalent as twisted links} if they are related by a finite sequence of isotopies of $\R^2$ and extended Reidemeister moves.  We denote it by $D\overset{\rm t}{\sim} D'$.  
A {\em twisted link} is an equivalence class of twisted link diagrams under the equivalence relation $\overset{\rm t}{\sim}$, \cite{Bou}.

\begin{figure}[H]
\begin{tabular}{cc}
\begin{minipage}{0.30\hsize}
\begin{center}
\includegraphics[height=2.0cm]{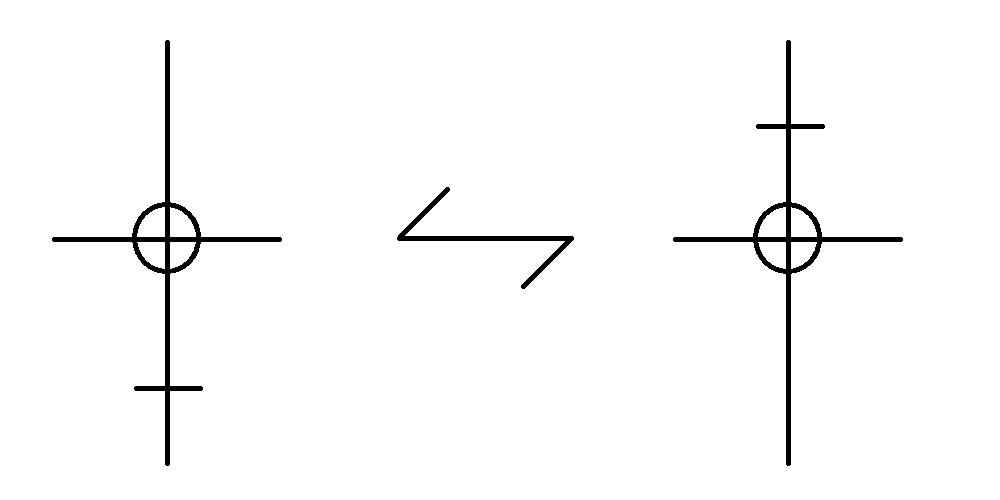}

${\rm T1}$
\end{center}
\end{minipage}
\begin{minipage}{0.30\hsize}
\begin{center}
\includegraphics[height=2.0cm]{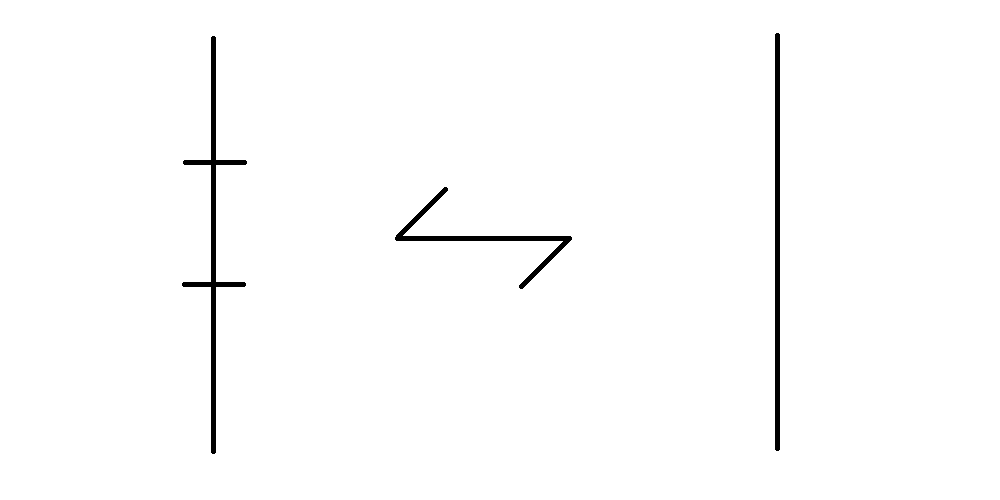}

${\rm T2}$
\end{center}
\end{minipage}
\begin{minipage}{0.35\hsize}
\begin{center}
\includegraphics[height=2.0cm]{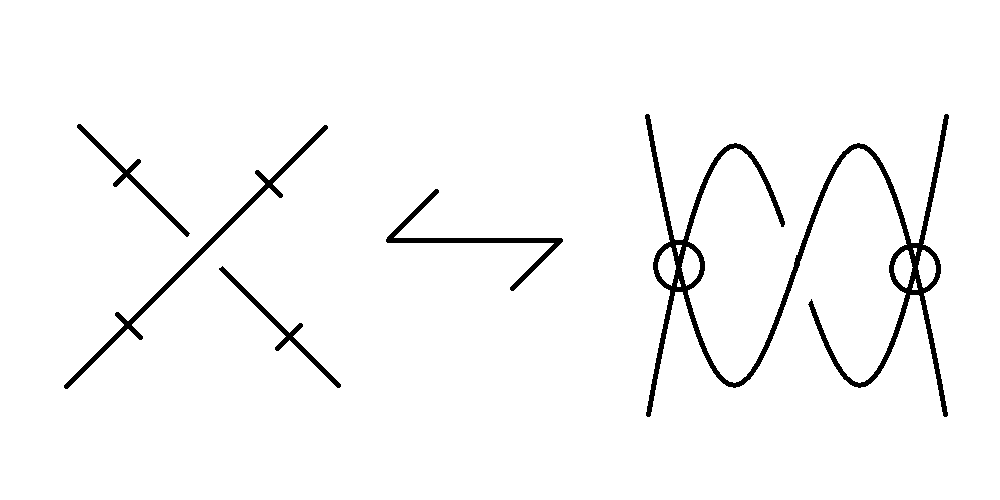}

${\rm T3}$
\end{center}
\end{minipage}
\end{tabular}
\vspace{0.1cm}
\caption{Twisted Reidemeister moves}
\label{fig:suri-T1-T3}
\end{figure}

Link diagrams are virtual link diagrams, and virtual link diagrams are twisted link diagrams. 
If two link diagrams are equivalent as links, i.e, related by a finite sequence of isotopies of $\R^2$ and Reidemeister moves R1, R2, R3, then they are equivalent as virtual links. If two virtual link diagrams are equivalent as virtual links, then they are equivalent as twisted links.  It is known that two link diagrams are equivalent as links if and only if they are equivalent as virtual links (or twisted links).

In this paper we assume that a virtual/twisted link diagram $D$ is {\em ordered}, i.e., the components are numbered by positive integers, and we consider that two virtual/twisted link diagrams are equivalent if they are equivalent with respect to the ordering.  

Let 
$D=D_1\cup\cdots\cup D_m$ be a twisted link diagram, where $D_i$ $(i=1, \dots m)$ is the $i$th component of $D$.

We say that the component $D_i$ is {\em of even type} if there are an even number of bars on $D_i$, or {\em of odd type} otherwise.  Note that if $D$ and $D'$ are equivalent then the parity of the $i$th component of $D$ is the same with that of $D'$.   

For example, the first component $D_1$ of the diagram below, $D=D_1\cup D_2$, is of even type and the second component $D_2$ is of odd type. 

\begin{figure}[htbp]
\begin{center}
\includegraphics[height=3.0cm]{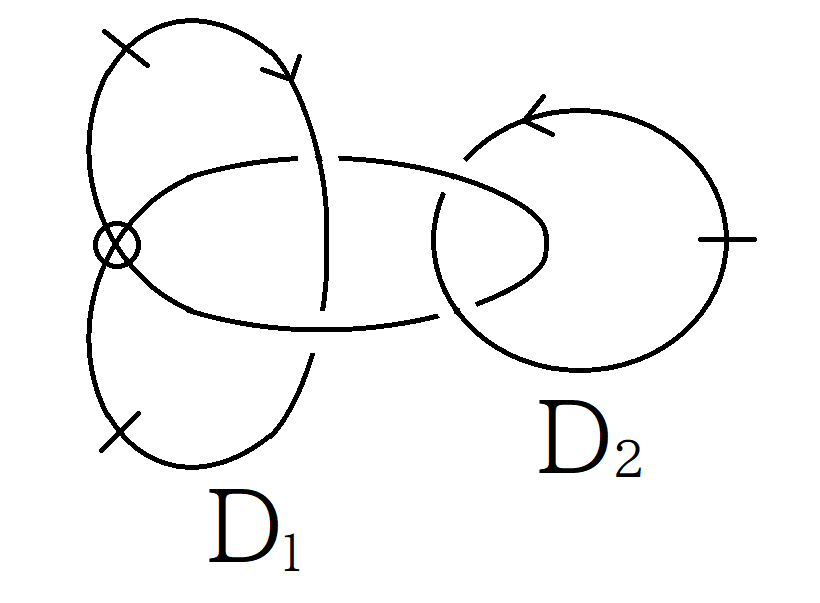}
\end{center}
\caption{$D=D_1\cup D_2$}
\label{fig:suri11}
\end{figure}

\section{Twisted intersection colorings}
\label{sect:twisted-coloring}

In this section we introduce the notion of a twisted intersection coloring.     

For a non-negative integer $n$, we denote by $\Z_n$ the cyclic group $\Z/n\Z$.  
An element $[k] = k + n \Z$ of $\Z_n$ is often denoted by its representative $k \in \Z$. 

Let $D=D_1\cup\ldots\cup D_m$ be a twisted link diagram.  
A {\em semiarc} of $D$ means an arc or a loop obtained by dividing the components of $D$ at all classical crossings and bars.  (A semiarc may have self-intersections, each of which is a virtual crossing of $D$. A semiarc is a loop if and only if it is a component of $D$ missing classical crossings and bars.)   

We denote by $\mathcal{A}(D_i)$ the set of semiarcs of $D$ on the $i$th component $D_i$.

\begin{definition}
Let $D=D_1\cup\ldots\cup D_m$ be a twisted link diagram.   
A {\em twisted intersection coloring mod $n$} on $D_i$ is a map 
$C:\mathcal{A}(D_i) \to \Z_n\times\{\pm 1\}$ 
satisfying the condition illustrated in Figure~\ref{fig:suri-C1B}, where we forget over/under information at classical crossings: 
\begin{itemize}
\item[(1)] 
When we pass through a classical crossing on $D_i$, 
the value $(a,e) \in \Z_n\times\{\pm 1\}$ changes into $(a+1,e)$ or $(a-1,e)$ as in Figure~\ref{fig:suri-C1B}, 
depending on the orientation of the arc we pass through at the crossing.   
\item[(2)] 
When we pass through a bar on $D_i$, the value $(a,e) \in \Z_n\times\{\pm 1\}$ changes into $(-a, -e)$. 
\end{itemize}
\begin{figure}[H]
\begin{center}
\includegraphics[height=2.5cm]{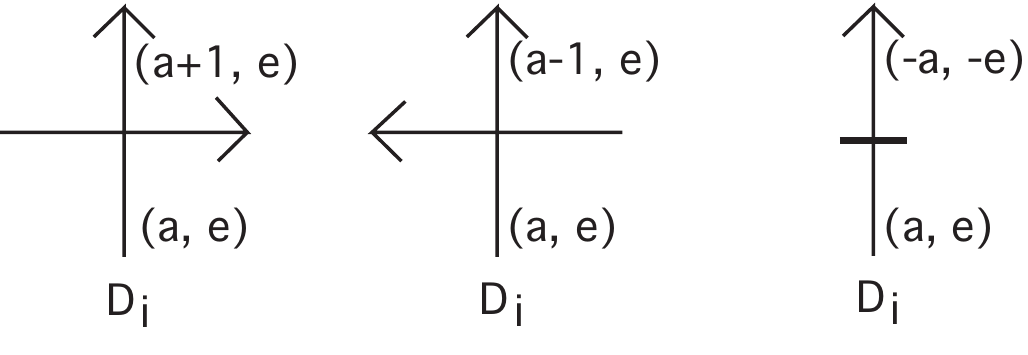}
\caption{Twisted intersection coloring}
\label{fig:suri-C1B}
\end{center}
\end{figure}
\end{definition}

In particular, a twisted intersection coloring mod $1$ on $D_i$ is an assignment of $(0,+1)$ or $(0,-1) \in \Z_1\times\{\pm 1\}$ to each semiarc on $D_i$ such that it changes at bars but does not at classical crossings. Thus, $D_i$ admits a twisted intersection coloring mod $1$  if and only if 
 $D_i$ is of even type.  In this case, there are two twisted intersection colorings mod $1$ on $D_i$. 

Suppose that $D_i$ is of even type.  Take a twisted intersection coloring mod $1$ on $D_i$ and fix it.  
Each semiarc on $D_i$ is colored with $(0,+1)$ or $(0,-1)$.   
We denote by 
$S_{L,+1}(D_i)$, $S_{L,-1}(D_i)$, $S_{R,+1}(D_i)$, and  $S_{R,-1}(D_i)$   
the set of classical crossings on $D_i$ as in Figure~\ref{fig:suri-d58B} respectively. 

\begin{figure}[H]
\begin{center}
\includegraphics[height=2.5cm]{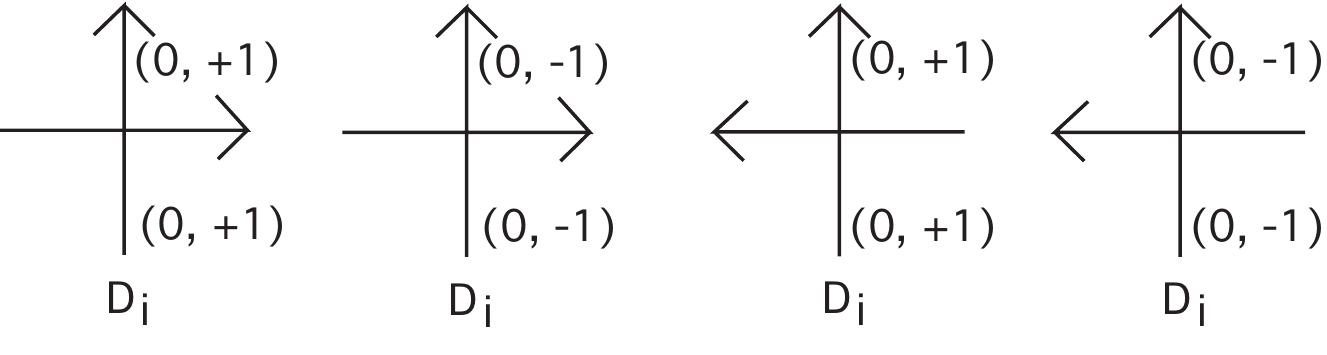}
\caption{$S_{L,+1}(D_i)$,  $S_{L,-1}(D_i)$,   $S_{R,+1}(D_i)$,  and  $S_{R,-1}(D_i)$}
\label{fig:suri-d58B}
\end{center}
\end{figure}

\begin{definition}
Let $D=D_1\cup\ldots\cup D_m$ be a twisted link diagram. Suppose that $D_i$ is of even type. 
The $i$th {\em twisted defect} of $D$ is 
a non-negative integer $d_i(D)$ defined by 
$$d_i(D)=| \# S_{L,+1}(D_i) - \# S_{L,-1}(D_i) - \# S_{R,+1}(D_i) + \# S_{R,-1}(D_i)|.$$
\end{definition}

The value $d_i(D)$ does not depend on a choice of the coloring mod $1$ on $D_i$.

\begin{proposition}
\label{prop:twisted_defect}
The $i$th twisted defect $d_i(D)$ is an invariant of a twisted link, namely, 
for two diagrams $D=D_1\cup\ldots\cup D_m$ and $D'=D'_1\cup\ldots\cup D'_m$ whose $i$th components are of even type, if $D \overset{\rm t}{\sim} D'$ then $d_i(D) = d_i(D')$.  
\end{proposition}

\begin{proof}
It is sufficient to consider a case where $D'$ is obtained from $D$ by an extended Reidemeister move and the 
twisted intersection colorings mod $1$ on $D_i$ and on $D'_i$ are identical outside of the region where the move is applied.

Suppose that $D'$ is obtained from $D$ by a move ${\rm R1}$ on a component $D_j$ of $D$ and the corresponding component $D'_j$ of $D'$ has an additional crossing, say $\tau$.  
Then $\tau \in S_{L,e}(D'_j) \cap S_{R,e}(D'_j)$ for some $e$. 
Hence $d_i(D) = d_i(D')$.  

Suppose that $D'$ is obtained from $D$ by a move ${\rm R2}$ on components $D_{j_1}$ and $D_{j_2}$.   
The corresponding components $D'_{j_1}$ and $D'_{j_2}$ of $D'$ has additional two crossings $\tau_1$ and $\tau_2$. 
If one of $\tau_1$ and $\tau_2$ belongs to 
$S_{L,e}(D'_i)$ then the other belongs to $S_{R,e}(D'_i)$. 
Hence  we have $d_i(D) = d_i(D')$.   

Suppose that $D'$ is obtained from $D$ by a move ${\rm R3}$ on components $D_{j_1}$, $D_{j_2}$, $D_{j_3}$ involving three crossings $\tau_1, \tau_2, \tau_3$ and the corresponding components $D'_{j_1}$, $D'_{j_2}$, $D'_{j_3}$ involves three crossings $\tau'_1, \tau'_2, \tau'_3$. Here we assume that if $\tau_1$ is a crossing between $D_{j_2}$ and $D_{j_3}$ then 
$\tau'_1$ is a crossing between $D'_{j_2}$ and $D'_{j_3}$, and so on.  If   
$\tau_1 \in S_{L,e}(D_i)$ then  $\tau'_1 \in S_{L,e}(D'_i)$ and if 
$\tau_1 \in S_{R,e}(D_i)$ then  $\tau'_1 \in S_{R,e}(D'_i)$, and so on. 
Hence  we have $d_i(D) = d_i(D')$.   

When $D'$ is obtained from $D$ by a move ${\rm V1}$, \dots, ${\rm V4}$, ${\rm T1}$ or ${\rm T2}$, it is obvious that $d_i(D) = d_i(D')$. 

Suppose that $D'$ is obtained from $D$ by a move ${\rm T3}$ and a crossing $\tau$ of $D$ becomes $\tau'$. 
If $\tau \in S_{L,e}(D_i)$ then $\tau' \in S_{R,-e}(D'_i)$, and 
if $\tau \in S_{R,e}(D_i)$ then $\tau' \in S_{L,-e}(D'_i)$.  Hence $d_i(D) = d_i(D')$.
\end{proof}

In what follows, ${\rm pr}_1:\Z_n\times\{\pm 1\}\to\Z_n$ and ${\rm pr}_2:\Z_n\times\{\pm 1\}\to\{\pm 1 \}$ denote the first and the second factor projections.

\begin{lemma}
\label{lem e2}
There exists a twisted intersection coloring mod $n$ on $D_i$ if and only if $D_i$ is of even type and $n|d_i(D)$. 
\end{lemma}

\begin{proof}
Take a semiarc $\alpha_0 \in \mathcal{A}(D_i)$ and name the semiarcs on $D_i$ with 
$\alpha_0, \alpha_1, \dots,$ $\alpha_{N-1}, \alpha_N (= \alpha_0)$ so that they appear on $D_i$ in this order along the direction of $D_i$.  
We consider a labeling on these semiarcs by elements of $\Z\times\{\pm 1\}$ 
as follows: Assign $(0,0)$ to $\alpha_0$, and assign elements of $\Z\times\{\pm 1\}$
to the semiarcs inductively as in Figure~\ref{fig:suri-C1B}.  Then the first factor of the label assigned $\alpha_N$ is 
$\# S_{L,+1}(D_i) - \# S_{L,-1}(D_i) - \# S_{R,+1}(D_i) + \# S_{R,-1}(D_i)$ and the second factor is $+1$ (or $-1$) if $D_i$ is of even (or odd) type.  There exists a twisted intersection coloring mod $n$ on $D_i$ if and only if 
$\# S_{L,+1}(D_i) - \# S_{L,-1}(D_i) - \# S_{R,+1}(D_i) + \# S_{R,-1}(D_i) = 0$ in $\Z_n$ and $D_i$ is of even type.  
\end{proof}

\begin{lemma}
\label{lem:e}
Let $D=D_1\cup\cdots\cup D_m$ be a twisted link diagram. Suppose that $D_i$ is of even type. 
Let 
$C_i$ and $C'_i$ be twisted intersection colorings mod $n$ on $D_i$. 
Suppose that 
${\rm pr}_2\circ C_i ={\rm pr}_2\circ C'_i$.  
Then there exists an integer $k$ mod $n$ such that for any $\alpha \in \mathcal{A}(D_i)$, the following condition is satisfied: 
\begin{itemize}
\item 
${\rm pr}_2\circ C_i(\alpha)=+1$ implies ${\rm pr}_1\circ C_i(\alpha)={\rm pr}_1\circ C'_i(\alpha)+k$, and 
\item
${\rm pr}_2\circ C_i(\alpha)=-1$ implies ${\rm pr}_1\circ C_i(\alpha)={\rm pr}_1\circ C'_i(\alpha)-k$.
\end{itemize}
\end{lemma}

\begin{proof}
Take a semiarc $\alpha_0 \in \mathcal{A}(D_i)$ and name the semiarcs on $D_i$ with 
$\alpha_0, \alpha_1, \dots,$ $\alpha_{N-1}, \alpha_N (= \alpha_0)$ so that they appear on $D_i$ in this order along the direction of $D_i$.  
Let $k = {\rm pr}_1\circ C_i(\alpha_0) - {\rm pr}_1\circ  C'_i(\alpha_0) \in \Z_n$.  By induction, we see that the desired condition is satisfied for all $s=0, 1, \dots, N$.    
\end{proof}

\begin{example}
Let $D=D_1$ be a twisted link diagram with one component depicted in Figure~\ref{fig:suri9B}. 
The twisted defect 
$d_1(D)$ is $0$.  The coloring 
$C_1:\mathcal{A}(D_1) \to \Z\times\{\pm 1\}$ depicted in Figure~\ref{fig:suri10B} is a twisted intersection coloring mod $0$ on $D_1$.  

\begin{figure}[H]
\begin{tabular}{cc}
\begin{minipage}{0.45\hsize}
\begin{center}
\includegraphics[height=3.5cm]{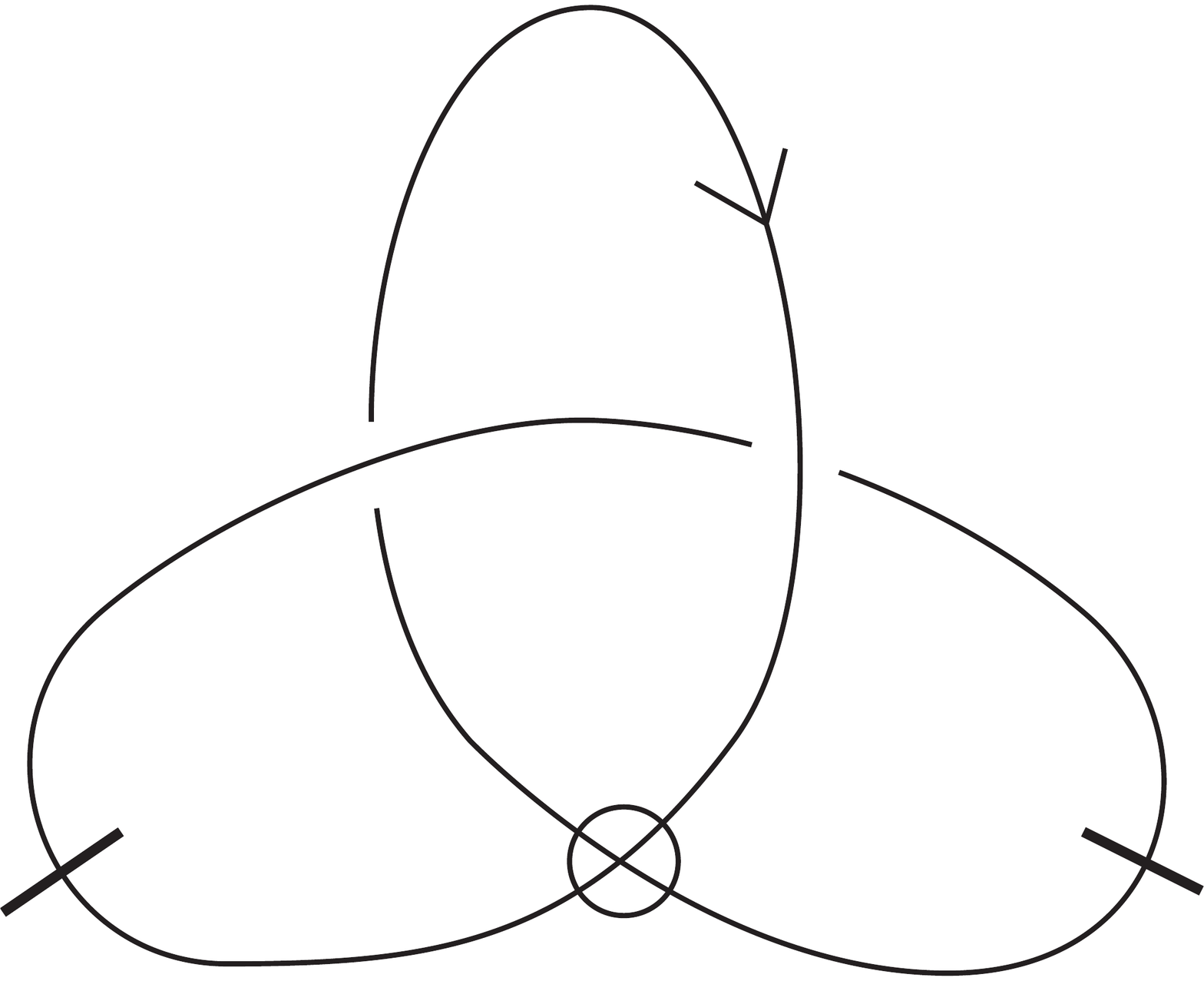}
\caption{$D=D_1$}
\label{fig:suri9B}
\end{center}
\end{minipage}
\begin{minipage}{0.50\hsize}
\begin{center}
\includegraphics[height=4.5cm]{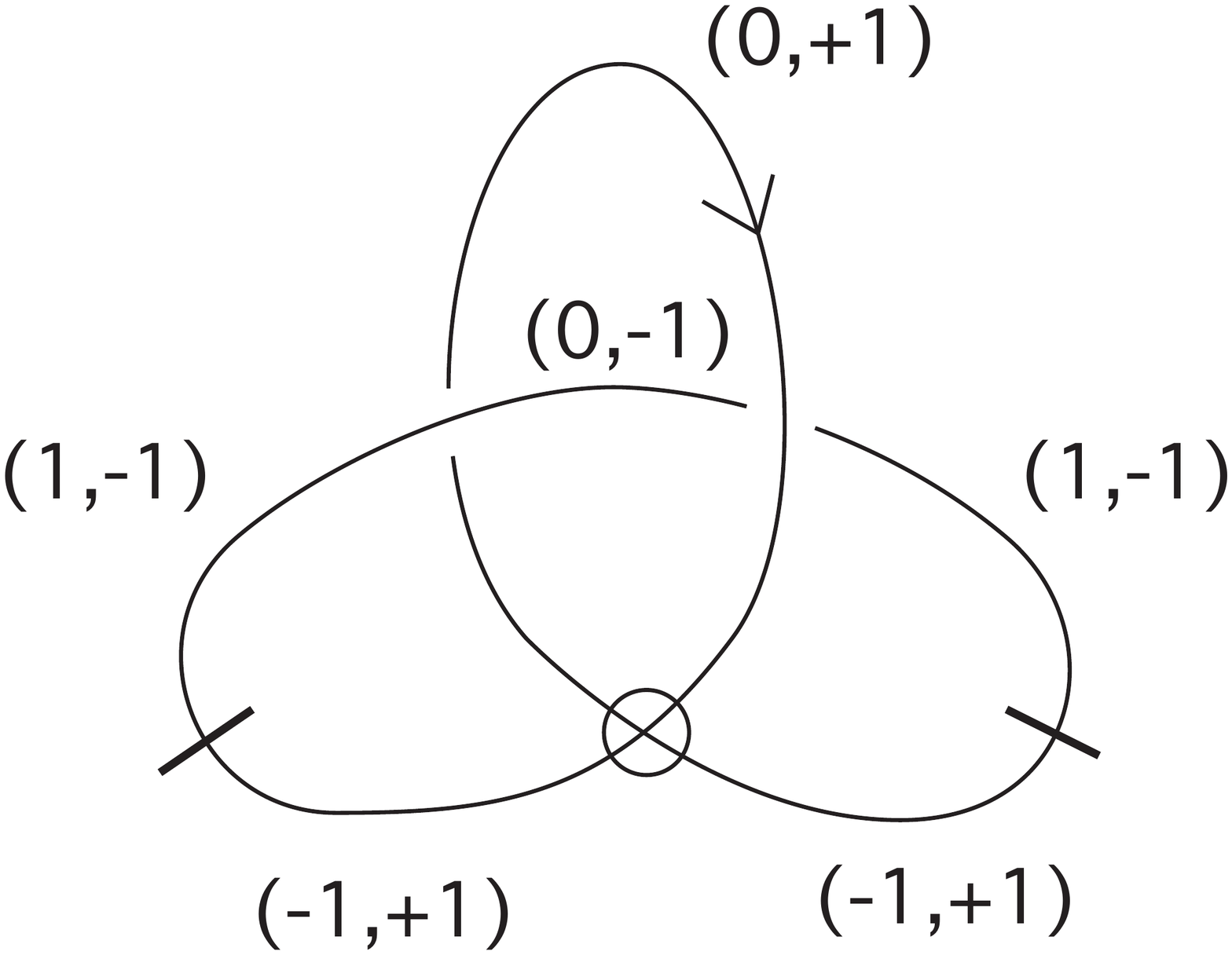}
\caption{A twisted intersection coloring mod~$0$}
\label{fig:suri10B}
\end{center}
\end{minipage}
\end{tabular}
\end{figure}

\end{example}

Let $D=D_1\cup\cdots\cup D_m$ be a twisted link diagram and $C_i$ be a twisted intersection coloring mod $n$ on $D_i$. 
Suppose that $D'=D'_1\cup\cdots\cup D'_m$ is a twisted link diagram obtained from $D$ by an extended Reidemeister move.  Then there exists a unique twisted intersection coloring mod $n$ on $D'_i$ such that $C_i$ and $C'_i$ are identical outside the region where the move is applied.  We say that {\em $C'_i$ is obtained from $C_i$ by the extended Reidemeister move}.

\begin{definition}
Let $D=D_1\cup\cdots\cup D_m$ and $D'=D'_1\cup\cdots\cup D'_m$ be twisted link diagrams. 
Let $C_i$ and $C'_i$ be twisted intersection colorings mod $n$ on $D_i$ and on $D'_i$, respectively.  
We say that $D$ with $C_i$ and $D'$ with $C'_i$ are {\em equivalent} if there exists a finite sequence of extended Reidemeister moves such that  both $D'$ and $C'_i$ are obtained from $D$ and $C_i$ by the sequence.  
\end{definition}

\section{Index polynomials and invariants}
\label{sect:invariants}

\begin{definition}
\label{def:twisted weight}
Let $D=D_1\cup\cdots\cup D_m$ be a twisted link diagram. 
Suppose that $D_i$ is of even type and let 
$C_i$ be a twisted intersection coloring mod $d_i(D)$ on $D_i$. 
For a classical crossing  $\tau$ which is a self-crossing of $D_i$,  
we define the {\em weight} $W_{(D_i, C_i)}(\tau)$ by 
$$W_{(D_i, C_i)}(\tau) = a-b  \in\Z_{d_i(D)},$$
where $a$ and $b$ are the first factors of the 
labels around the crossing $\tau$ as in Figure~\ref{fig:suri-wB}. 

\begin{figure}[H]
\begin{center}
\includegraphics[height=2.0cm]{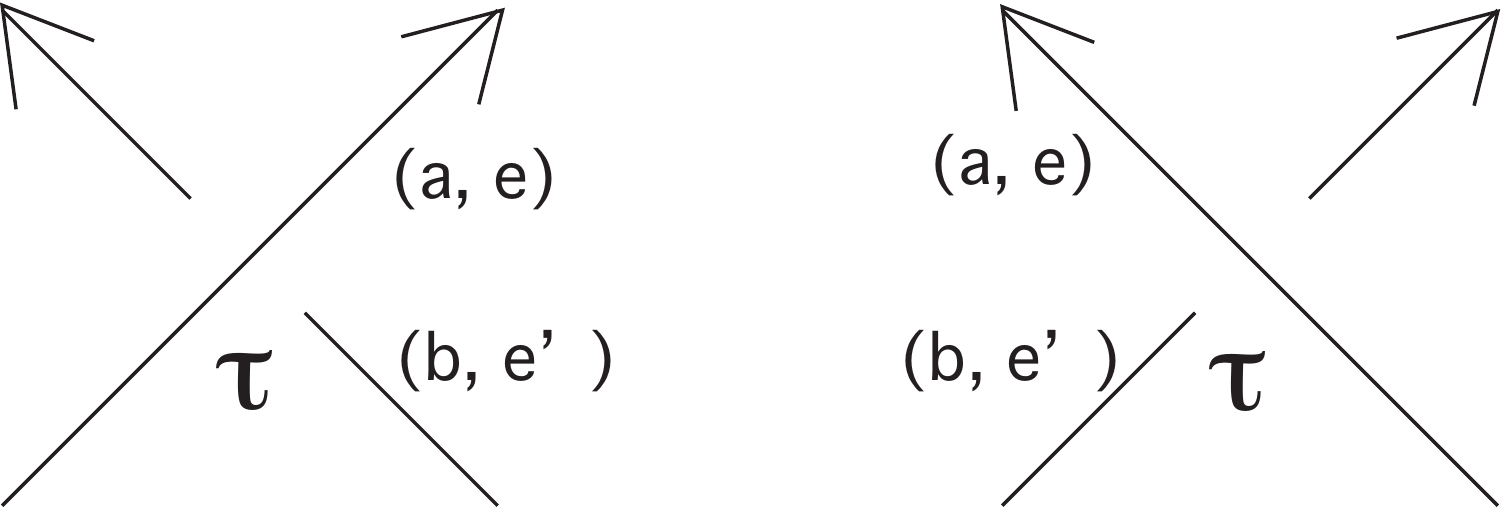}
\end{center}
\caption{$W_{(D_i, C_i)}(\tau) = a-b$}
\label{fig:suri-wB}
\end{figure}

\end{definition}

Let $D_i$ be a component of even type and $C_i$ be a twisted intersection coloring mod $d_i(D)$ on $D_i$.  
We denote by 
$T^{(D_i,C_i)}$, $T_{+1}^{(D_i,C_i)}$ and $T_{-1}^{(D_i,C_i)}$ the sets of self-crossings of $D_i$ such that the colors around $\tau$ are as in Figure~\ref{fig:suri-3-5B} respectively:   

$\tau \in T^{(D_i,C_i)}$ $\Leftrightarrow$ The second factors of the over and the under arcs are the same. 

$\tau \in T_{+1}^{(D_i,C_i)}$ $\Leftrightarrow$ The second factor of the over arc is $+1$ and that of the under is $-1$. 

$\tau \in T_{-1}^{(D_i,C_i)}$ $\Leftrightarrow$ The second factor of the over arc is $-1$ and that of the under is $+1$.

\begin{figure}[H]
\begin{center}
\includegraphics[height=1.8cm]{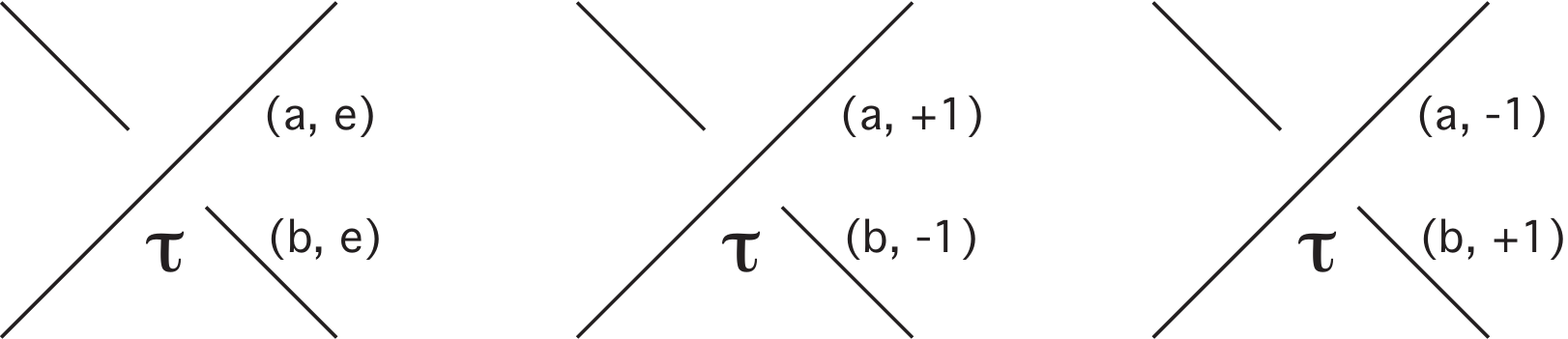}
\end{center}
\caption{$T^{(D_i,C_i)}$, $T_{+1}^{(D_i,C_i)}$ and $T_{-1}^{(D_i,C_i)}$}
\label{fig:suri-3-5B}
\end{figure}

\begin{definition}
\label{t ind}

Let $D_i$ be a component of even type and $C_i$ be a twisted intersection coloring mod $d_i(D)$ on $D_i$.  
The {\em index polynomials} 
$\psi^{(D_i,C_i)}(t)$, $\psi_{+1}^{(D_i,C_i)}(t)$ and 
$\psi_{-1}^{(D_i,C_i)}(t)$ are polynomials in 
$\Z[t^{\pm 1}]/(t^{d_i(D)}-1)$ defined as follows:  
\begin{eqnarray}
\psi^{(D_i,C_i)}(t)  &  = &  \sum_{\tau\in T^{(D_i,C_i)}} {\rm sign}(\tau)(t^{W_{(D_i,C_i)}(\tau)}-1), \\
\psi_{+1}^{(D_i,C_i)}(t)  &  = & \sum_{\tau\in T_{+1}^{(D_i,C_i)}} {\rm sign}(\tau)t^{W_{(D_i,C_i)}(\tau)},  \\
\psi_{-1}^{(D_i,C_i)}(t)  & =  &  \sum_{\tau\in T_{-1}^{(D_i,C_i)}} {\rm sign}(\tau)t^{W_{(D_i,C_i)}(\tau)}. 
\end{eqnarray}
\end{definition}

\begin{example}
Let $D=D_1$ be a twisted link diagram with $d_1(D)=0$ and let  
$C_1:\mathcal{A}(D_1) \to \Z\times\{\pm 1\}$ be a twisted intersection coloring mod $0$  
depicted in Figure~\ref{fig:suri10-2B}.   
Then 
$\tau_1\in T_{+1}^{(D_1,C_1)}$, $W_{(D_1,C_1)}(\tau_1)=-1$, 
$\tau_2\in T_{-1}^{(D_1,C_1)}$ and $W_{(D_1,C_1)}(\tau_2)=1$. 
Hence 
$\psi^{(D_1,C_1)}(t)=0$, $\psi_{+1}^{(D_1,C_1)}(t)=t^{-1}$ and $\psi_{-1}^{(D_1,C_1)}(t)=t$.  

\begin{figure}[H]
\begin{center}
\includegraphics[height=3.0cm]{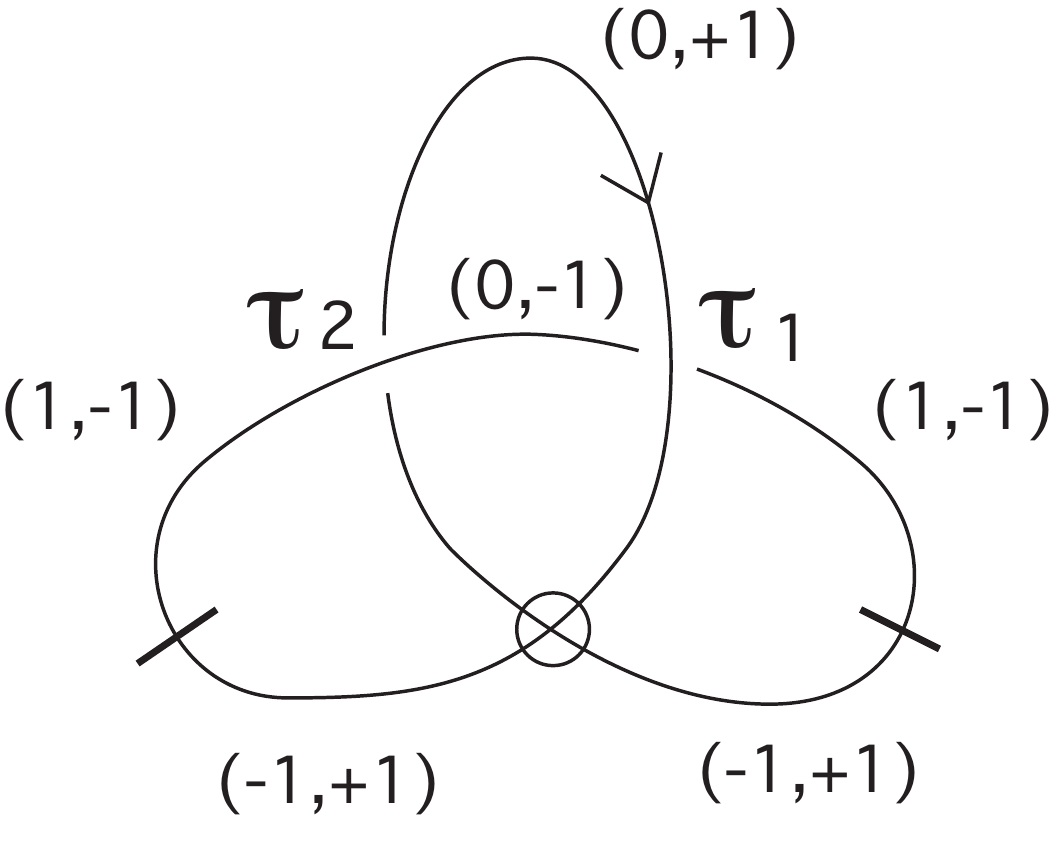}
\caption{A twisted intersection coloring mod $0$}
\label{fig:suri10-2B}
\end{center}
\end{figure}

\end{example}

\begin{theorem}
\label{thm:index-colored-twisted-link}
Let $D=D_1\cup\cdots\cup D_m$ and $D'=D'_1\cup\cdots\cup D'_m$ be twisted link diagrams such that 
$D_i$ and $D'_i$ are of even type for some $i$. 
Let $C_i$ and $C'_i$ be twisted intersection colorings mod  $d_i(D)$ and $d_i(D')$ on $D_i$ and $D'_i$, respectively.  
If $D$ with $C_i$ are equivalent to $D'$ with $C'_i$, then 
\begin{eqnarray}
\psi^{(D_i,C_i)}(t) &=& \psi^{(D'_i,C'_i)}(t), \label{eqn:thm-psi}\\ 
\psi_{+1}^{(D_i,C_i)}(t) &=& \psi_{+1}^{(D'_i,C'_i)}(t) \quad \mbox{and} \label{eqn:thm-psi0} \\ 
\psi_{-1}^{(D_i,C_i)}(t) &=& \psi_{-1}^{(D'_i,C'_i)}(t).  \label{eqn:thm-psi1}
\end{eqnarray} 
\end{theorem}

\begin{proof}
It is sufficient to consider a case where $D'$ with $C'_i$ are obtained from $D$ with $C_i$ by an extended Reidemeister move.  
Moreover, for ${\rm R1}, {\rm R2}, {\rm R3}$, it is sufficient to consider moves 
${\rm R1a}$, ${\rm R1b}$, ${\rm R2a}$, ${\rm R3a}$ depicted 
in Figure~\ref{fig:Polyak}, which are a generating set for oriented Reidemeister moves~\cite{Pol}.  

\begin{figure}[H]
\begin{tabular}{cc}
\begin{minipage}{0.18\hsize}
\begin{center}
\includegraphics[height=1.5cm]{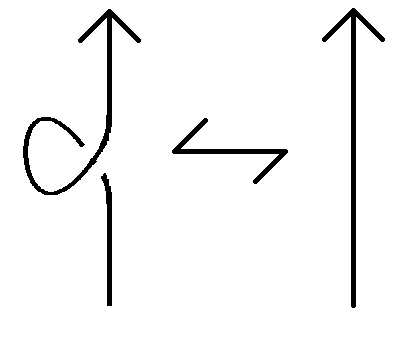}

R1a
\end{center}
\end{minipage}
\begin{minipage}{0.18\hsize}
\begin{center}
\includegraphics[height=1.5cm]{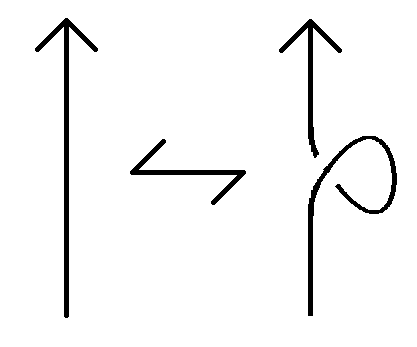}

R1b
\end{center}
\end{minipage}
\begin{minipage}{0.30\hsize}
\begin{center}
\includegraphics[height=1.5cm]{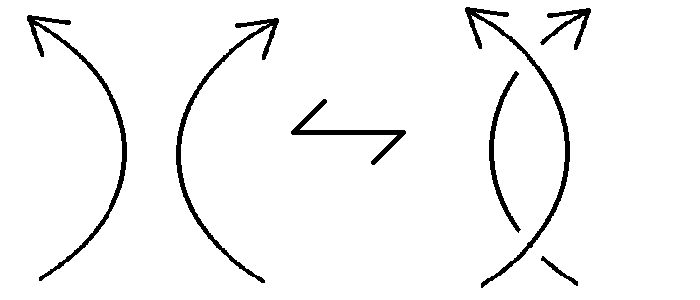}

R2a
\end{center}
\end{minipage}
\begin{minipage}{0.30\hsize}
\begin{center}
\includegraphics[height=1.5cm]{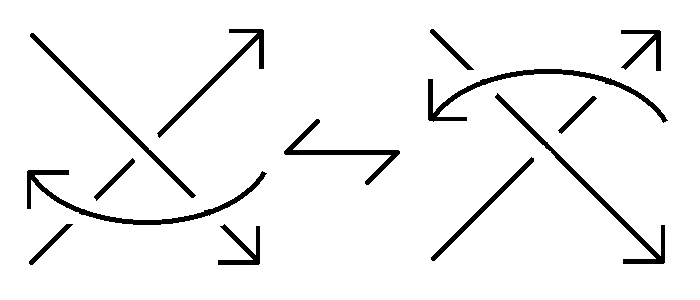}

R3a
\end{center}
\end{minipage}
\end{tabular}
\vspace{0.2cm}
\caption{A generating set of oriented Reidemeister moves}
\label{fig:Polyak}
\end{figure}

Suppose that the move is ${\rm R1a}$ or ${\rm R1b}$ on $D_k$ as 
depicted in Figure~\ref{fig:suri-R1a-c-R1b-c-R2a-c}.  
Since $\tau_1 \notin T_{+1}^{(D_i,C_i)} \cup T_{-1}^{(D_i,C_i)}$, 
Equalities~(\ref{eqn:thm-psi0}) and (\ref{eqn:thm-psi1}) are obvious.  
If $i \neq k$, then $\tau_1 \notin T^{(D_i,C_i)}$ and we have 
Equality~(\ref{eqn:thm-psi}).  
Assume $i=k$. The semiarcs around $\tau_1$ are labeled as in the figure, 
where we drop the second factor of the labels.  
Since $W_{(D_i,C_i)}(\tau_1)=0$, we have 
${\rm sign}(\tau_1)(t^{W_{(D_i,C_i)}(\tau_1)}-1) = 0$, which implies 
Equality~(\ref{eqn:thm-psi}).  

\begin{figure}[H]
\begin{tabular}{cc}

\begin{minipage}{0.27\hsize}
\begin{center}
\includegraphics[height=2.0cm]{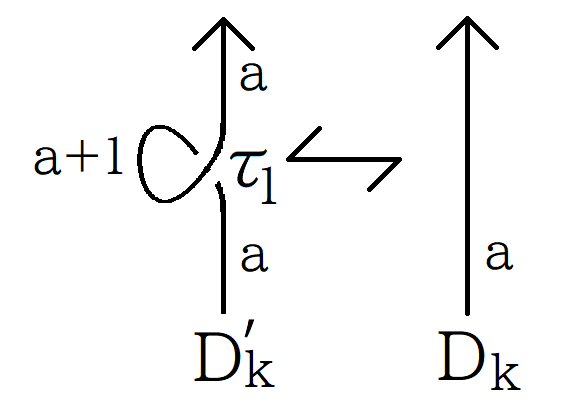}
\end{center}
\end{minipage}

\begin{minipage}{0.27\hsize}
\begin{center}
\includegraphics[height=2.0cm]{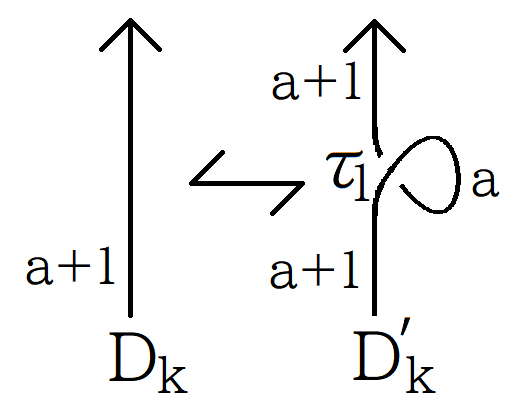}
\end{center}
\end{minipage}

\begin{minipage}{0.35\hsize}
\begin{center}
\includegraphics[height=2.3cm]{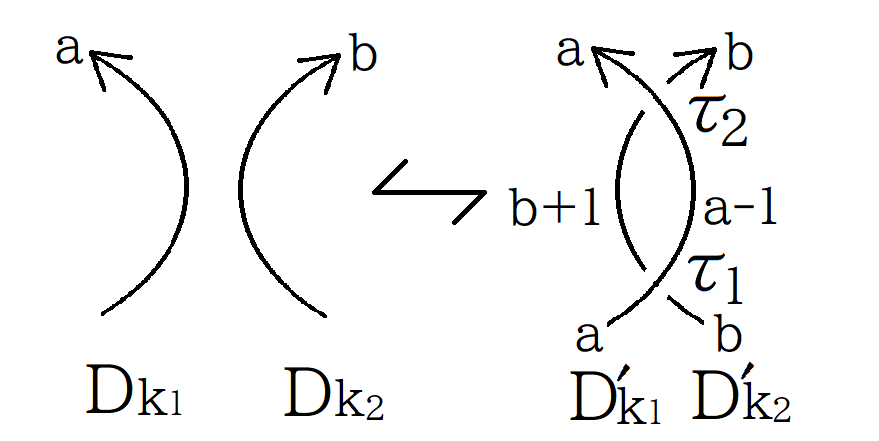}
\end{center}
\end{minipage}

\end{tabular}
\vspace{0.2cm}
\caption{R1a, R1b and R2a}
\label{fig:suri-R1a-c-R1b-c-R2a-c}
\end{figure}

Suppose that the move is ${\rm R2a}$, as depicted in Figure~\ref{fig:suri-R1a-c-R1b-c-R2a-c}. 
If $i \neq k_1$ or $i \neq k_2$, then $\tau_1$ and $\tau_2$ are not self-crossings of $D_i$. Thus we have Equalities~(\ref{eqn:thm-psi}), (\ref{eqn:thm-psi0}) and (\ref{eqn:thm-psi1}). 

Assume $i = k_1 = k_2$. The semiarcs around $\tau_1$ and $\tau_2$ are labeled as in the figure, 
where we drop the second factor of the labels.  
Both $\tau_1$ and $\tau_2$ belong to one of $T^{(D_i,C_i)}$, $T_{+1}^{(D_i,C_i)}$ and $T_{+1}^{(D_i,C_i)}$.  
Since ${\rm sign}(\tau_1)=1, {\rm sign}(\tau_2)=-1$ and 
$W_{(D_i,C_i)}(\tau_1)=W_{(D_i,C_i)}(\tau_2)$, we have Equalities~(\ref{eqn:thm-psi}), (\ref{eqn:thm-psi0}) and (\ref{eqn:thm-psi1}). 

Suppose that the move is ${\rm R3a}$, as depicted in Figure~\ref{fig:suri-R3a-c}. 
For each $k\in\{1,2,3\}$, 
if $\tau_k$ belongs to $T^{(D_i,C_i)}$, $T_{+1}^{(D_i,C_i)}$ or $T_{-1}^{(D_i,C_i)}$, then 
$\tau'_k$ belongs to $T^{(D'_i,C'_i)}$, $T_{+1}^{(D'_i,C'_i)}$ or $T_{-1}^{(D'_i,C'_i)}$, respectively.  
 Since  
${\rm sign}(\tau_k)={\rm sign}(\tau'_k)$ and $W_{(D_i,C_i)}(\tau_k)=W_{(D'_i,C'_i)}(\tau'_k)$, we have 
Equalities~(\ref{eqn:thm-psi}), (\ref{eqn:thm-psi0}) and (\ref{eqn:thm-psi1}).  
\begin{figure}[H]
\begin{center}
\includegraphics[height=2.7cm]{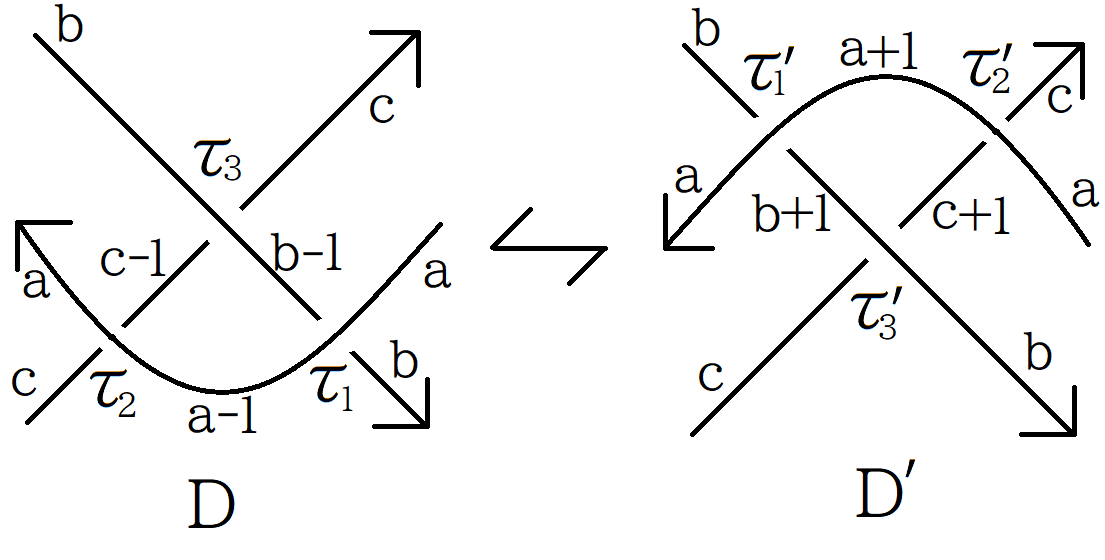}
\caption{R3a}
\label{fig:suri-R3a-c}
\end{center}
\end{figure}

When the move is ${\rm V1}, \dots, {\rm V4}$, ${\rm T1}$ or ${\rm T2}$, Equalities~(\ref{eqn:thm-psi}), (\ref{eqn:thm-psi0}) and (\ref{eqn:thm-psi1}) are obvious.   

Suppose that the move is ${\rm T3}$ as depicted in Figure~\ref{fig:sr-6B}. If $i \neq j$ or $i \neq k$, then Equalities~(\ref{eqn:thm-psi}), (\ref{eqn:thm-psi0}) and (\ref{eqn:thm-psi1}) are obvious.   
Now we assume that $i=j=k$. 
\begin{figure}[H]
\begin{center}
\includegraphics[height=3.2cm]{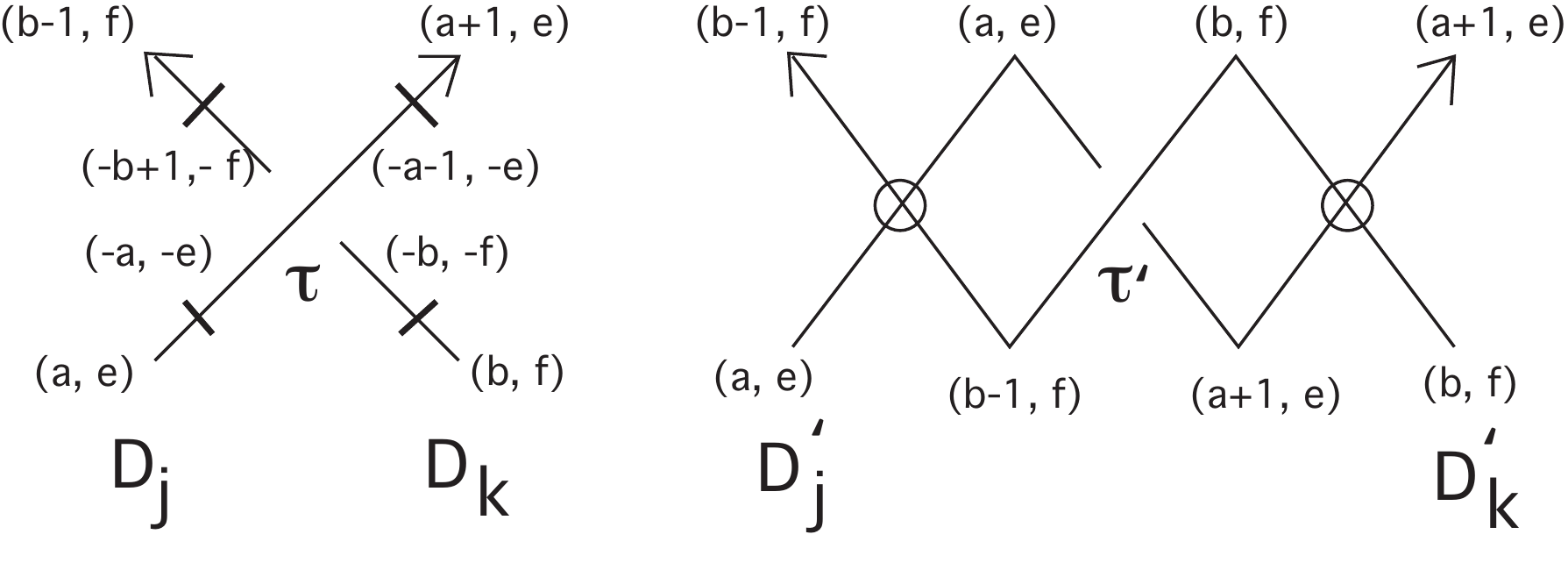}
\end{center}
\caption{}
\label{fig:sr-6B}
\end{figure}

If $\tau$ belongs to $T^{(D_i,C_i)}$, $T_{+1}^{(D_i,C_i)}$ or $T_{-1}^{(D_i,C_i)}$, then 
$\tau'$ belongs to $T^{(D'_i,C'_i)}$, $T_{+1}^{(D'_i,C'_i)}$ or $T_{-1}^{(D'_i,C'_i)}$, respectively.  
Since ${\rm sign}(\tau)= {\rm sign}(\tau')$ and 
$W_{(D_i,C_i)}(\tau)=W_{(D'_i,C'_i)}(\tau')$, we have 
Equalities~(\ref{eqn:thm-psi}), (\ref{eqn:thm-psi0}) and (\ref{eqn:thm-psi1}).  
The cases of other possible orientations on ${\rm T3}$ are shown similarly. 
\end{proof}

\begin{definition}
Let $f_1,f_2,g_1,g_2$ be elements of $\Z[t^{\pm 1}]/(t^n-1)$, where $n$ is a non-negative integer. 
We say that two pairs $(f_1,f_2)$ and $(g_1,g_2)$ are {\em $2$-equivalent} and denote it by  
$(f_1,f_2)\overset{(2)}{\equiv}(g_1,g_2)$ if  
$(f_1,f_2)=(g_1\cdot t^{2k},~g_2\cdot t^{-2k})$ or $(f_1,f_2)=(g_2\cdot t^{2k},~g_1\cdot t^{-2k})$
for some integer $k$.  
\end{definition}

The following theorem is one of our main results, which states that the index polynomial $\psi^{(D_i,C_i)}(t)$ and 
the $2$-equivalence class of $(\psi_{+1}^{(D_i,C_i)}(t),~\psi_{-1}^{(D_i,C_i)}(t))$ are invariants of a twisted link.

\begin{theorem}
\label{thm b}
Let $D=D_1\cup\cdots\cup D_m$ and $D'=D'_1\cup\cdots\cup D'_m$ be twisted link diagrams such that 
$D_i$ and $D'_i$ are of even type for some $i$. 
Let $C_i$ and $C'_i$ be twisted intersection colorings mod  $d_i(D)$ and $d_i(D')$ on $D_i$ and $D'_i$, respectively. 
If $D $ and $D'$ are equivalent, then we have 
\begin{eqnarray}
\psi^{(D_i,C_i)}(t) &=& \psi^{(D'_i,C'_i)}(t),  \\ 
(\psi_{+1}^{(D_i,C_i)}(t),~\psi_{-1}^{(D_i,C_i)}(t)) &\overset{(2)}{\equiv}& (\psi_{+1}^{(D'_i,C'_i)}(t),~\psi_{-1}^{(D'_i,C'_i)}(t)).
\end{eqnarray}
\end{theorem}

\begin{proof}
It follows from Theorem~\ref{thm:index-colored-twisted-link} and the following lemmas (Lemmas~\ref{lem:g} and \ref{lem:f}). 
\end{proof}

\begin{lemma}
\label{lem:g}

$D=D_1\cup\cdots\cup D_m$ be a twisted link diagram.  Suppose that $D_i$ is of even type.  Let 
$C_i$ and $C'_i$ be twisted intersection colorings mod $d_i(D)$ on $D_i$. 
Suppose that 
${\rm pr}_2\circ C_i = - {\rm pr}_2\circ C'_i$ and 
${\rm pr}_1\circ C_i = {\rm pr}_1\circ C'_i$.  Then 
\begin{eqnarray}
\psi^{(D_i,C_i)}(t) &=& \psi^{(D_i,C'_i)}(t), \\
\psi_{+1}^{(D_i,C_i)}(t) &=& \psi_{-1}^{(D_i,C'_i)}(t), \\ 
\psi_{-1}^{(D_i,C_i)}(t) &=& \psi_{+1}^{(D_i,C'_i)}(t).
\end{eqnarray}
\end{lemma}

\begin{proof}
Since ${\rm pr}_2\circ C_i = - {\rm pr}_2\circ C'_i$, we have 
$T^{(D_i,C_i)}=T^{(D_i,C'_i)}$, $T_{+1}^{(D_i,C_i)}=T_{-1}^{(D_i,C'_i)}$ and  $T_{-1}^{(D_i,C_i)}=T_{+1}^{(D_i,C'_i)}$.  
Since ${\rm pr}_1\circ C_i = {\rm pr}_1\circ C'_i$, for any self-crossing $\tau$ of $D_i$,  
$W_{(D_i,C_i)}(\tau) = W_{(D_i,C'_i)}(\tau)$.  
Thus we obtain the assertion. 
\end{proof}

\begin{lemma}
\label{lem:f}
Let 
$D=D_1\cup\cdots\cup D_m$ be a twisted link diagram.  Suppose that $D_i$ is of even type.  Let 
$C_i$ and $C'_i$ be twisted intersection colorings mod $d_i(D)$  on $D_i$. 
Suppose that 
${\rm pr}_2\circ C_i ={\rm pr}_2\circ C'_i $.  
Let $k$ be an integer mod  $d_i(D)$ as in Lemma~\ref{lem:e}.  Then 
\begin{eqnarray}
\psi^{(D_i,C_i)}(t) &=& \psi^{(D_i,C'_i)}(t), \\ 
\psi_{+1}^{(D_i,C_i)}(t) &=& \psi_{+1}^{(D_i,C'_i)}(t)\cdot t^{2k}, \\ 
\psi_{-1}^{(D_i,C_i)}(t) &=& \psi_{-1}^{(D_i,C'_i)}(t)\cdot t^{-2k}.
\end{eqnarray}
\end{lemma}

\begin{proof}
Since  
${\rm pr}_2\circ C_i ={\rm pr}_2\circ C'_i$, we have 
$T^{(D_i,C_i)}=T^{(D_i,C'_i)}$, 
$T_{+1}^{(D_i,C_i)}=T_{+1}^{(D_i,C'_i)}$ and 
$T_{-1}^{(D_i,C_i)}=T_{-1}^{(D_i,C'_i)}$.
By Lemma~\ref{lem:e}, for any $\gamma \in\mathcal{A}(D_i)$, 
\begin{itemize}
\item 
${\rm pr}_2\circ C_i(\gamma)=+1$ implies ${\rm pr}_1\circ C_i(\gamma)={\rm pr}_1\circ C'_i(\gamma)+k$, and 
\item
${\rm pr}_2\circ C_i(\gamma)=-1$ implies ${\rm pr}_1\circ C_i(\gamma)={\rm pr}_1\circ C'_i(\gamma)-k$.
\end{itemize}

Let $\tau$ be a self-crossing of $D_i$ and let $\alpha$ and $\beta$ be semiarcs as in Figure~\ref{fig:sr-w3B}. 

\begin{figure}[H]
\begin{center}
\includegraphics[height=1.8cm]{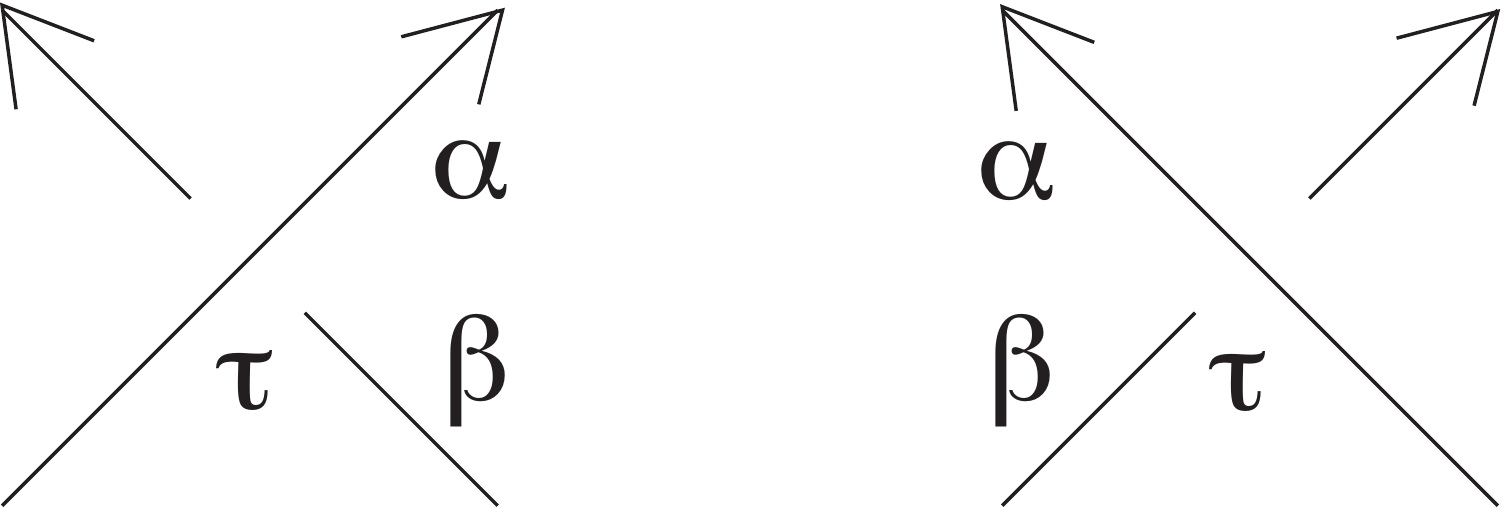}
\caption{}
\label{fig:sr-w3B}
\end{center}
\end{figure}

(1) Suppose that $\tau\in T^{(D_i,C_i)}$.  Then  
$({\rm pr}_2\circ C_i)(\alpha)=({\rm pr}_2\circ C_i)(\beta)$.  
Thus we have 
\begin{align*}
W_{(D_i,C_i)}(\tau) &= {\rm pr}_1\circ C_i(\alpha)-{\rm pr}_1\circ C_i(\beta) \\
                             &= ({\rm pr}_1\circ C'_i(\alpha) \pm k)-({\rm pr}_1\circ C'_i(\beta) \pm k)\\
                             &= {\rm pr}_1\circ C'_i(\alpha)-{\rm pr}_1\circ C'_i(\beta)\\
                             &= W_{(D_i,C'_i)}(\tau)
\end{align*}
This implies $\psi^{(D_i,C_i)}(t)=\psi^{(D_i,C'_i)}(t)$.

(2) Suppose that  $\tau\in T_{+1}^{(D_i,C_i)}$. Then 
$({\rm pr}_2\circ C_i)(\alpha)=+1$ and $({\rm pr}_2\circ C_i)(\beta)=-1$. 
Thus we have 
\begin{align*}
W_{(D_i,C_i)}(\tau) &= {\rm pr}_1\circ C_i(\alpha)-{\rm pr}_1\circ C_i(\beta) \\
                             &= ({\rm pr}_1\circ C'_i(\alpha)+k)-({\rm pr}_1\circ C'_i(\beta)-k)\\
                             &= {\rm pr}_1\circ C'_i(\alpha)-{\rm pr}_1\circ C'_i(\beta)+2k\\
                             &= W_{(D_i,C'_i)}(\tau)+2k.  
\end{align*}
This implies 
$\psi_{+1}^{(D_i,C_i)}(t) =  \psi_{+1}^{(D_i,C'_i)}(t)\cdot t^{2k}$.

(3) Suppose that $\tau\in T_{-1}^{(D_i,C_i)}$. Then 
$({\rm pr}_2\circ C_i)(\alpha)=-1$ and $({\rm pr}_2\circ C_i)(\beta)=+1$.  Thus we have 
\begin{align*}
W_{(D_i,C_i)}(\tau) &= {\rm pr}_1\circ C_i(\alpha)-{\rm pr}_1\circ C_i(\beta) \\
                             &= ({\rm pr}_1\circ C'_i(\alpha)-k)-({\rm pr}_1\circ C'_i(\beta)+k)\\
                             &= {\rm pr}_1\circ C'_i(\alpha)-{\rm pr}_1\circ C'_i(\beta)-2k\\
                             &= W_{(D_i,C'_i)}(\tau)-2k
\end{align*}
This implies  
$\psi_{-1}^{(D_i,C_i)}(t) = \psi_{-1}^{(D_i,C'_i)}(t)\cdot t^{-2k}$. 
\end{proof}

\begin{example}
Let $D=D_1$ and $D'=D'_1$ be twisted link diagrams depicted in Figure~\ref{fig:sr-46-47}. 
Our invariant can distinguish them as follows: 
$d_1(D)=d_1(D')=2$.  For some twisted intersection colorings $C_1$ and $C'_1$ mod $2$ on $D_1$ and $D'_1$, we have  
$(\psi_{+1}^{(D_1,C_1)}(t),~\psi_{-1}^{(D_1,C_1)}(t))=(0,1)$ and $(\psi_{+1}^{(C'_1,D'_1)}(t),~\psi_{-1}^{(C'_1,D'_1)}(t))=(0,-1)$ 
in $( \Z[t^{\pm 1}]/(t^2-1) )^2$.  
Since $(0,1)\overset{(2)}{\not\equiv}(0,-1)$, we see that $D$ and $D'$ are not equivalent.

\begin{figure}[H]
\begin{tabular}{cc}
\begin{minipage}{0.40\hsize}
\begin{center}
\includegraphics[height=2.0cm]{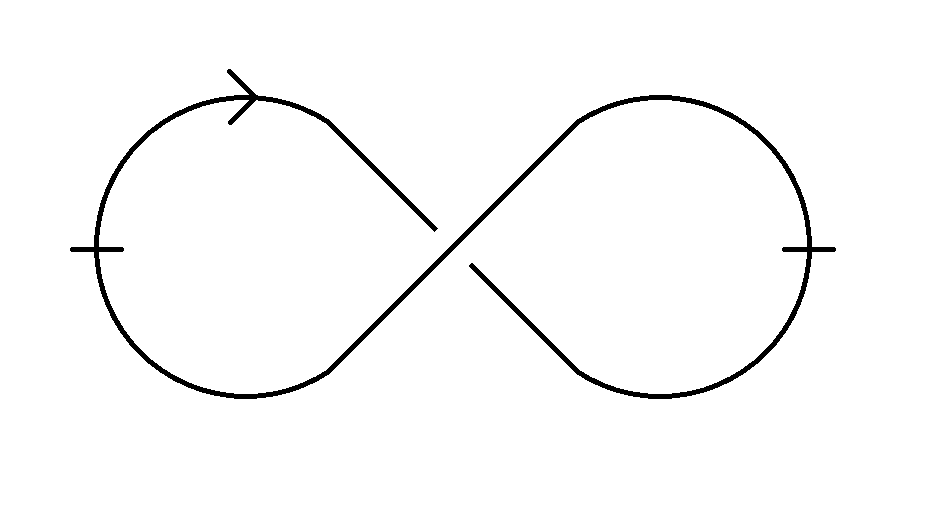}

$D=D_1$
\end{center}
\end{minipage}

\begin{minipage}{0.40\hsize}
\begin{center}
\includegraphics[height=2.0cm]{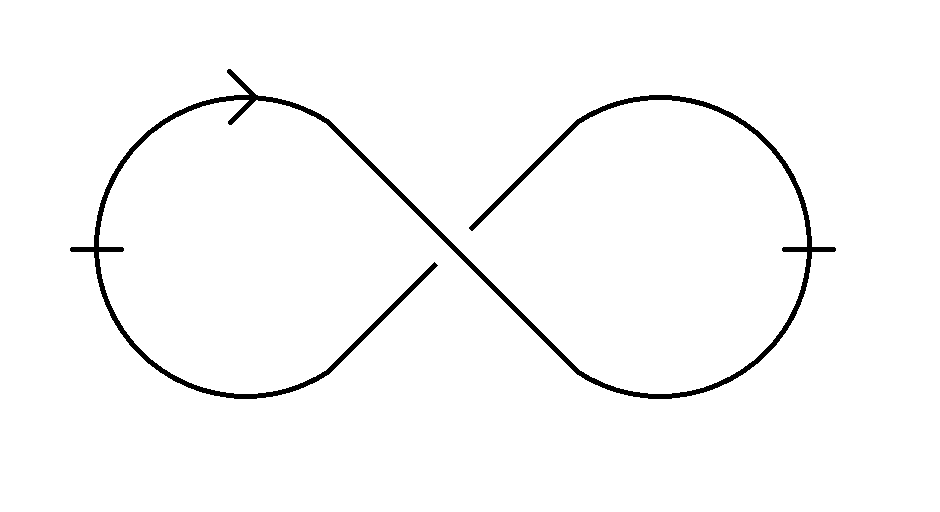}

$D'=D'_1$
\end{center}
\end{minipage}
\end{tabular}
\caption{}
\label{fig:sr-46-47}
\end{figure}

\end{example}

N.~Kamada \cite{K13} defined an invariant of a twisted link which is 
called the index polynomial invariant and denoted by $\widetilde{Q}_D(t)$. 
The invariant $\widetilde{Q}_D(t)$ does not distinguish the diagrams $D$ and $D'$  
in Figure~\ref{fig:sr-46-47}: $\widetilde{Q}_D(t)=\widetilde{Q}_{D'}(t)=0$. 
Thus our invariant, the $2$-equivalence class of $(\psi_{+1}, \psi_{-1})$, is different from hers.

\section{Double coverings} 
\label{sect:double}

The double covering of a twisted link was originally defined in \cite{KK16} using orientation double coverings of surfaces.  
Roughly speaking, a twisted link diagram $D$ is associated with a link diagram $D_{\Sigma}$ on a compact (or closed) surface $\Sigma$.  
Consider an orientation double covering $p: \widetilde{\Sigma} \to \Sigma$ of the surface $\Sigma$, and let $\widetilde{D_{\Sigma}}$ be the link diagram on $\widetilde{\Sigma}$ which is the pull back of $D_{\Sigma}$.  The link diagram $\widetilde{D_{\Sigma}}$ on $\widetilde{\Sigma}$ is associated with a virtual link diagram $\widetilde{D}$, which is called a {\em double covering diagram} of $D$.  A double covering diagram $\widetilde{D}$ of $D$ 
is uniquely determined up to moves ${\rm V1}, \dots, {V4}$. 

For a twisted link diagram $D=D_1\cup\cdots\cup D_m$, we denote by 
$\widetilde{D}=\widetilde{D}_1\cup\cdots\cup\widetilde{D}_m$ 
a double covering diagram of $D$ such that $\widetilde{D}_i$ is the preimage of $D_i$. 
When $D_i$ is of odd type, $\widetilde{D}_i$ is a single component of $\widetilde{D}$.  
When $D_i$ is of even type, $\widetilde{D}_i$ consists of two components of $\widetilde{D}$.  Thus, in general, the virtual link diagram 
$\widetilde{D}=\widetilde{D}_1\cup\cdots\cup\widetilde{D}_m$  is not ordered but it is {\it partitioned} 
 with $\{1, \dots, m\}$, i.e., each component is labeled by integers $1,\dots, m$.   

\begin{theorem}[\cite{KK16}]
\label{thm:doubleA}
If two twisted link diagrams $D$ and $D'$ are equivalent as (ordered) twisted links, then their double covering diagrams $\widetilde{D}$  and $\widetilde{D}'$ are equivalent as (partitioned) virtual links.  
\end{theorem}

A practical method of constructing a double covering diagram of a twisted link diagram is introduced in \cite{KK16}.  Let us recall the method here.   

Let $D$ be a twisted link diagram. 

(1) Move $D$ by an isotopy of $\R^2$ into the half plane $\{(x,y)\in\R^2~|~x>0\}$ such that 
the $y$-coordinates of the classical crossings, virtual crossings and bars are all distinct, and  
these $y$-coordinates avoid the critical values of the map from the underlying immersed loops of $D$ in $\R^2$ to the $y$-axis.  

We denote by $s(D)$ the twisted link diagram in $\{(x,y)\in\R^2~|~x<0\}$ which is obtained from $D$ by applying the reflection along the $y$-axis and then switching over/under information at all classical crossings. 
The union $s(D) \cup D$ is a twisted link diagram in $\R^2$. See the left of Figure~\ref{fig:suri-20-arrow-17}. 

\begin{figure}[H]
\begin{tabular}{cc}

\begin{minipage}{0.40\hsize}
\begin{center}
\includegraphics[height=3.7cm]{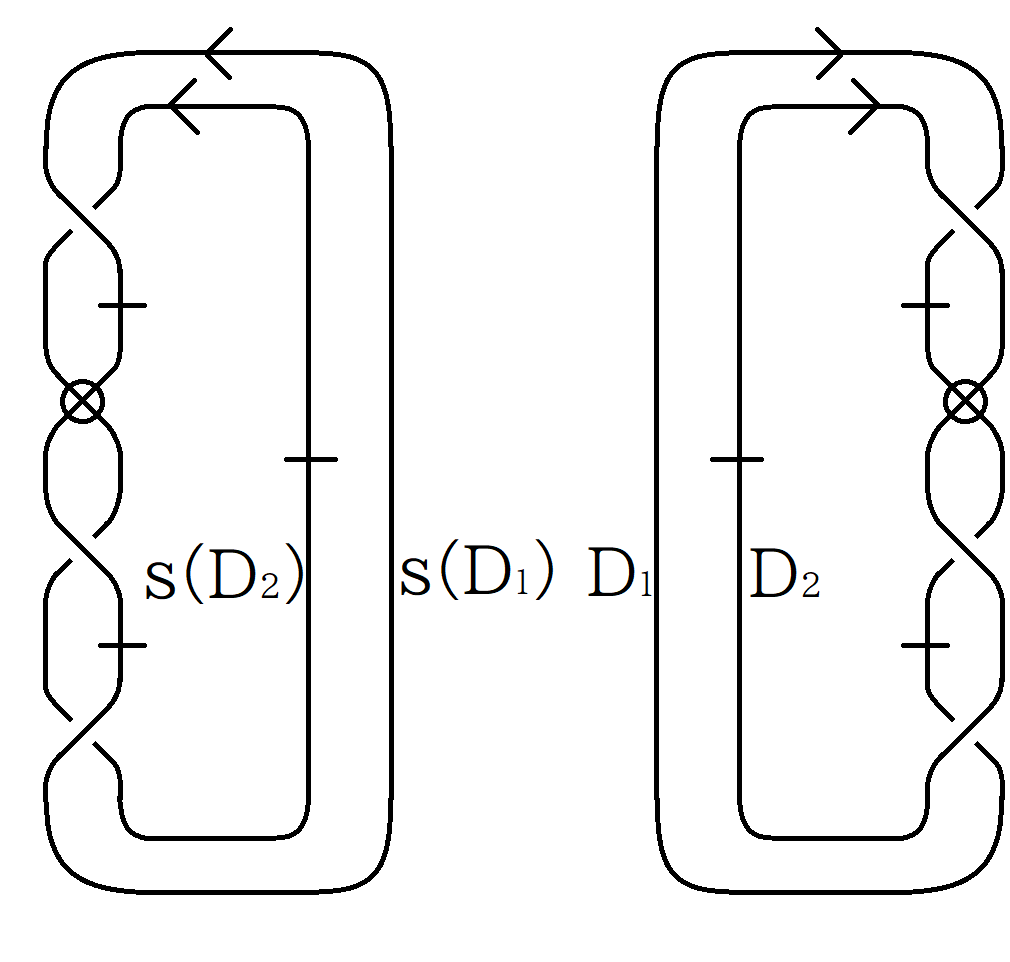}

$s(D) \cup D$
\end{center}

\end{minipage}
\begin{minipage}{0.20\hsize}
\begin{center}
\includegraphics[height=1.0cm]{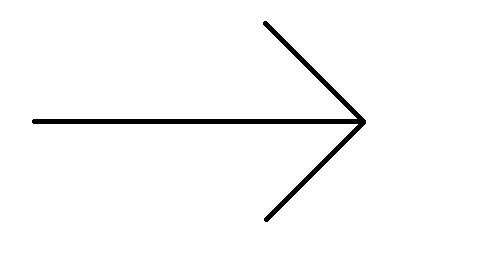}
\end{center}
\end{minipage}

\begin{minipage}{0.40\hsize}
\begin{center}
\includegraphics[height=3.7cm]{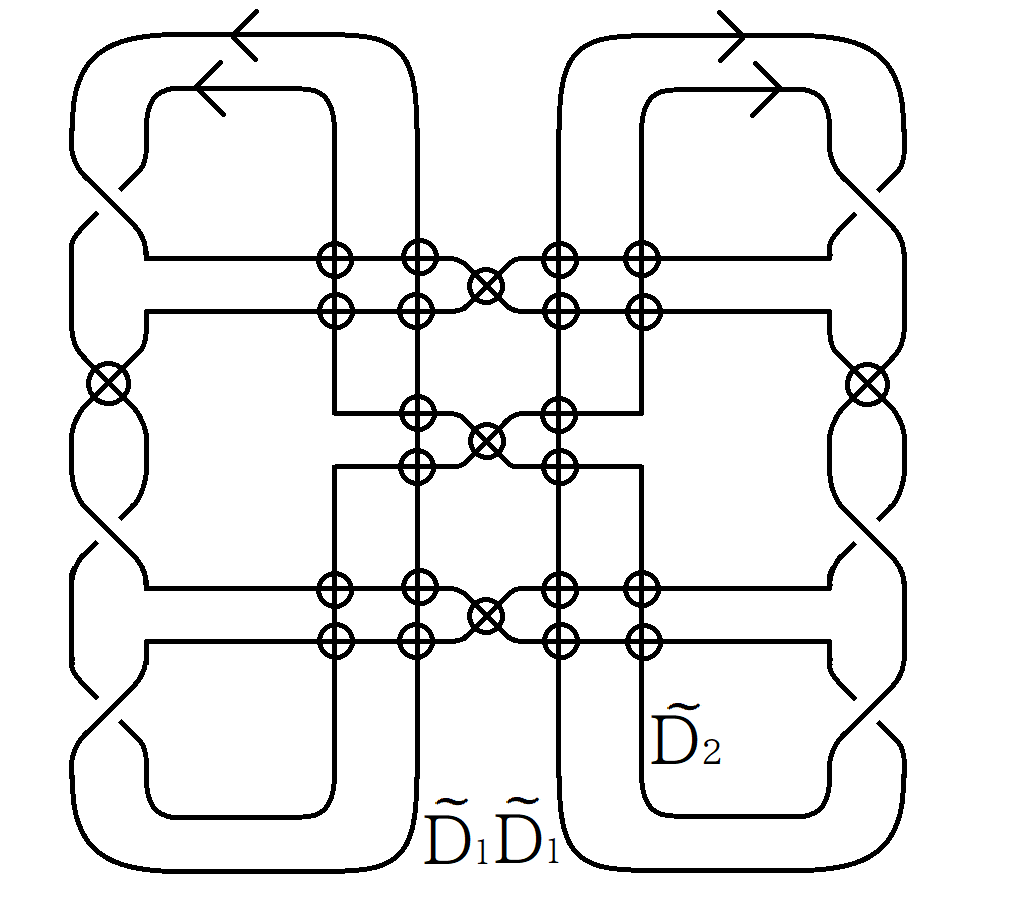}

$\widetilde{D}=\widetilde{D}_1\cup\widetilde{D}_2$
\end{center}
\end{minipage}
\end{tabular}
\caption{}
\label{fig:suri-20-arrow-17}
\end{figure}

(2) For each bar of $D$ and the corresponding bar of $s(D)$, consider a horizontal line segment connecting them, and let $N$ be a regular neighborhood of the segment in $\R^2$.  The intersection $N \cap (s(D) \cup D)$ consists of vertical line seguments as in Figure~\ref{fig:suri19}. 
Cut open the two line segments with bars and connect them with two arcs having virtual crossings as in the figure.

\begin{figure}[H]
\begin{center}
\includegraphics[height=2.5cm]{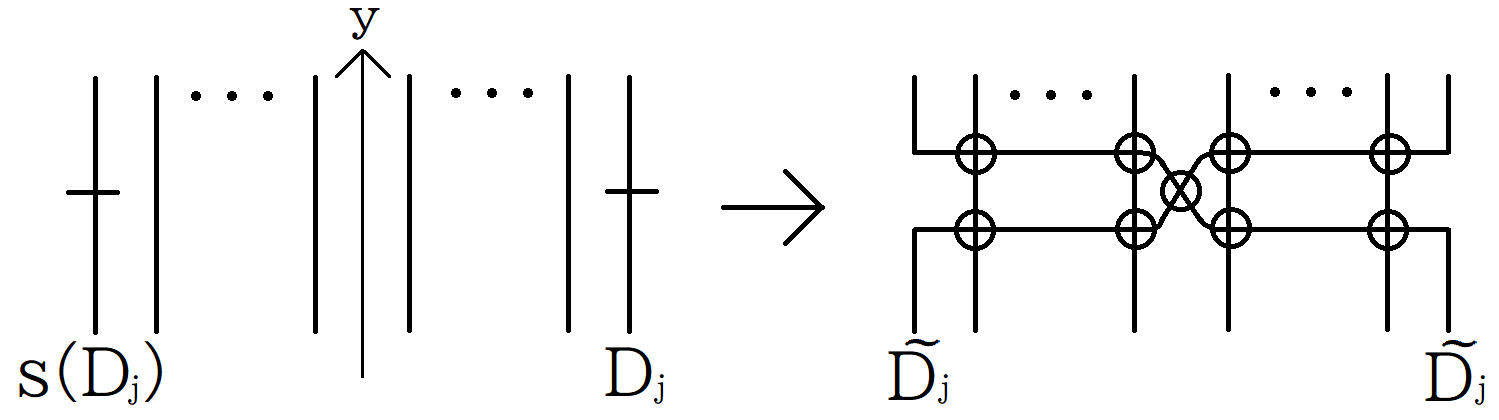}
\caption{Cut open $D$ at the bars on $s(D)$ and $D$ to obtain $\widetilde{D}$}
\label{fig:suri19}
\end{center}
\end{figure}

Now we obtain a virtual link diagram, which is a double covering diagram $\widetilde{D}$ of $D$. See Figure~\ref{fig:suri-20-arrow-17}.

\begin{theorem}
\label{thm:doubleB}
There exists infinitely many pairs of twisted links such that for each pair the two twisted links are distinct but their double coverings are equivalent as virtual links. 

Precisely speaking, 
for each positive integer $n$, there exits a pair of 
twisted link diagrams $(D(n)^x, D(n)^y)$ such that (1) they are not equivalent as (ordered or unordered) twisted links, (2) their double covering diagrams are equivalent as (partitioned or unpartitioned) virtual links, and (3)    
the twisted links represented by $D(n)^x$ and $D(n)^y$ are different from those by $D(n')^x$ and $D(n')^y$ for $n\neq n'$. 
\end{theorem}

This theorem implies that the converse of Theorem~\ref{thm:doubleA} is not true.  

\begin{proof}
For each positive integer $n$, let 
$D(n)^x=D(n)_1^x\cup D(n)_2^x$ and $D(n)^y=D(n)_1^y\cup D(n)_2^y$ be twisted link diagrams depicted in 
Figure~\ref{fig:sr-39C}, where the box labeled with $n$ stands for $2n-1$ half twists as in Figure~\ref{fig:sr-39F}. 

\begin{figure}[H]
\begin{center}
\includegraphics[width=10cm]{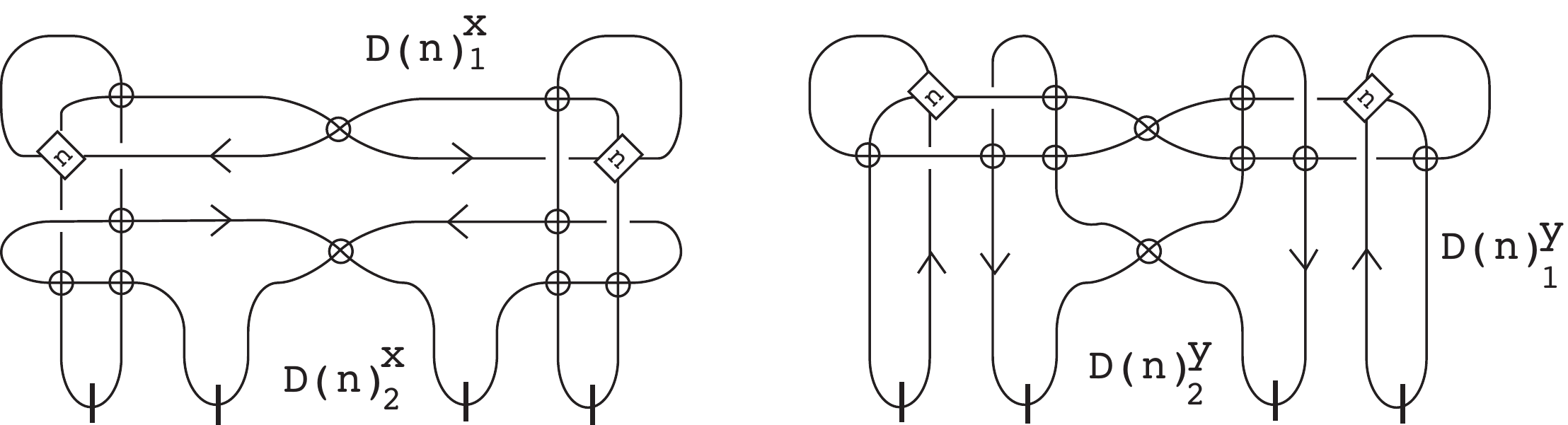}
\caption{$D(n)^x=D(n)_1^x\cup D(n)_2^x$ and $D(n)^y=D(n)_1^y\cup D(n)_2^y$}
\label{fig:sr-39C}
\end{center}
\end{figure}

\begin{figure}[H]
\begin{center}
\includegraphics[height=2.0cm]{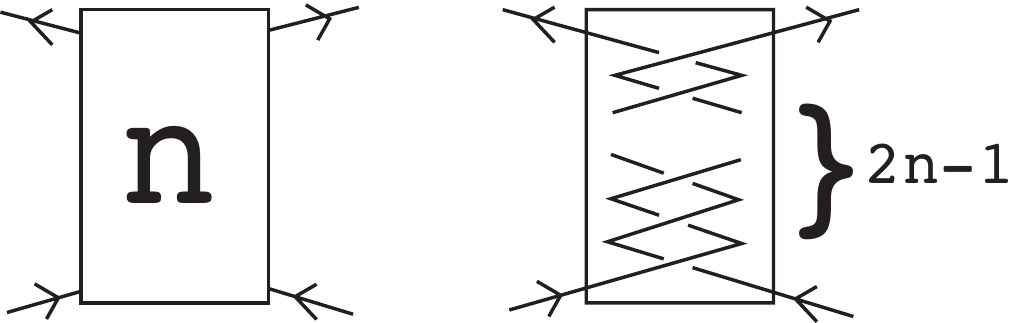}
\caption{$2n-1$ half twists}
\label{fig:sr-39F}
\end{center}
\end{figure}

By a direct computation, we see that $d_1(D(n)^x)=d_1(D(n)^y)=2$.  
Let $C(n)_1^x$ and $C(n)_1^y$ be twisted intersection colorings mod $2$ on $D(n)_1^x$ and $D(n)_1^y$  
as in Figures~\ref{fig:sr-39EB} and~\ref{fig:sr-39EC}, respectively.  

\begin{figure}[H]
\begin{center}
\includegraphics[width=13.0cm]{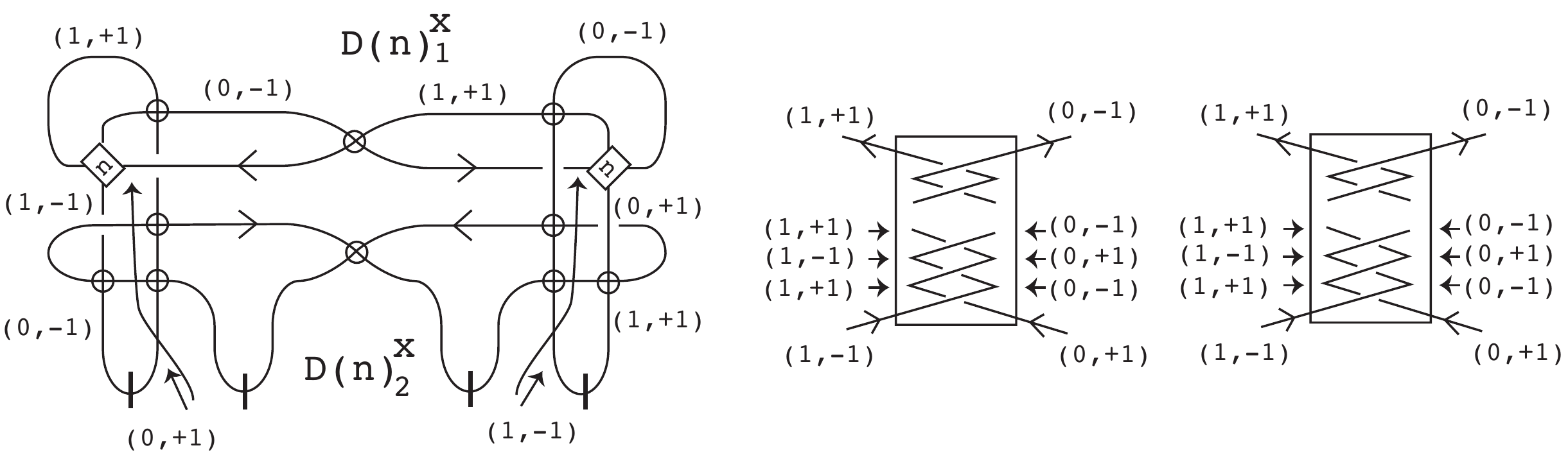}
\caption{$C(n)_1^x$, a twisted intersection coloring mod $2$ on $D(n)_1^x$}
\label{fig:sr-39EB}
\end{center}
\end{figure}

\begin{figure}[H]
\begin{center}
\includegraphics[width=13.0cm]{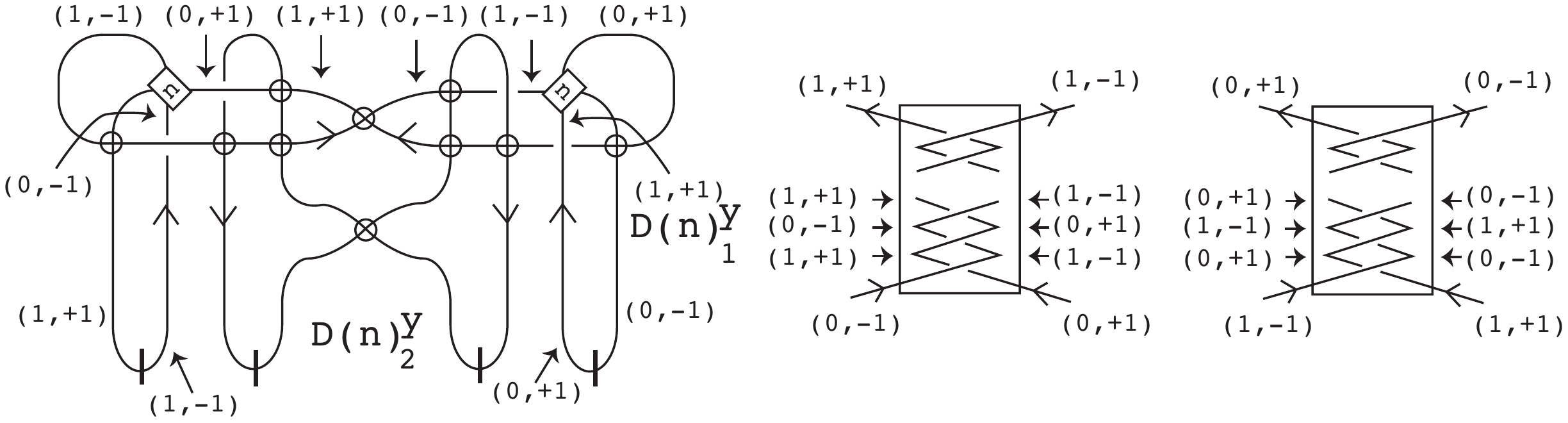}
\caption{$C(n)_1^y$, a twisted intersection coloring mod $2$ on $D(n)_1^y$}
\label{fig:sr-39EC}
\end{center}
\end{figure}

There exist $4n$ self-crossings of $D(n)_1^x$.  Two of them belong to $T^{(D(n)_1^x,C(n)_1^x)}$, 
$2(n-1)$ of them belong to $T_{+1}^{(D(n)_1^x,C(n)_1^x)}$ and $2n$ of them belong to $T_{-1}^{(D(n)_1^x,C(n)_1^x)}$.  
Let $\tau$ be any self-crossing belonging to $T_{+1}^{(D(n)_1^x,C(n)_1^x)}$ or $T_{-1}^{(D(n)_1^x,C(n)_1^x)}$, which is a crossing appearing in the $2n-1$ half twists.  
Then ${\rm sign}(\tau)=1$ and $W_{(D(n)_1^x,C(n)_1^x)}(\tau)=0 \in \Z_2$.  
Therefore we have 
$\psi_{+1}^{(D(n)_1^x,C(n)_1^x)}(t)=2(n-1)$ and $\psi_{-1}^{(D(n)_1^x,C(n)_1^x)}(t)=2n$ in $\Z[t^{\pm 1}]/(t^2-1)$.  
Hence 
$(\psi_{+1}^{(D(n)_1^x,C(n)_1^x)}(t), \psi_{-1}^{(D(n)_1^x,C(n)_1^x)}(t)) = ( 2(n-1), 2n)$.  
 
On the other hand, 
there exist $4n$ self-crossings of $D(n)_1^y$.  Two of them belong to $T^{(D(n)_1^y,C(n)_1^y)}$, 
$2(n-1)$ of them belong to $T_{+1}^{(D(n)_1^y,C(n)_1^y)}$ and $2n$ of them belong to $T_{-1}^{(D(n)_1^y,C(n)_1^y)}$.  
Let $\tau$ be any self-crossing belonging to $T_{+1}^{(D(n)_1^y,C(n)_1^y)}$ or $T_{-1}^{(D(n)_1^y,C(n)_1^y)}$. 
Then ${\rm sign}(\tau)=1$ and $W_{(D(n)_1^y,C(n)_1^y)}(\tau)=1 \in \Z_2$.  
Therefore we have 
$\psi_{+1}^{(D(n)_1^y,C(n)_1^y)}(t)=2(n-1)t$ and $\psi_{-1}^{(D(n)_1^y,C(n)_1^y)}(t)=2nt$ in $\Z[t^{\pm 1}]/(t^2-1)$. 
Hence 
$(\psi_{+1}^{(D(n)_1^y,C(n)_1^y)}(t), \psi_{-1}^{(D(n)_1^y,C(n)_1^y)}(t)) = ( 2(n-1)t, 2nt)$.  

Since 
$(\psi_{+1}^{(D(n)_1^x,C(n)_1^x)}(t), \psi_{-1}^{(D(n)_1^x,C(n)_1^x)}(t)) \overset{(2)}{\not\equiv}(\psi_{+1}^{(D(n)_1^y,C(n)_1^y)}(t), \psi_{-1}^{(D(n)_1^y,C(n)_1^y)}(t))$,  
by Theorem~\ref{thm b} we see that 
$D(n)^x$ and $D(n)^y$ are not equivalent as ordered twisted links.

By a direct computation, we see that $d_2(D(n)^x)=2$.  
Since there is no self-crossing on $D(n)_2^x$, we have 
$\psi_{+1}^{(D(n)_2^x,C(n)_2^x)}(t)= \psi_{-1}^{(C(n)_2^x,D_2^x)}(t)=0$ in $\Z[t^{\pm 1}]/(t^2-1)$. 
Since $(0,0) \overset{(2)}{\not\equiv} (2(n-1)t, 2nt)$, by Theorem~\ref{thm b} we see that 
$D(n)_2^x \cup D(n)_1^x$ and $D^y(n)$ are not equivalent as ordered twisted links. 
Therefore $D^x(n)$ and $D^y(n)$ are not equivalent as (ordered or unordered) twisted links. 

Using the invariants, we also see that 
the pairs of twisted links represented by $D(n)^x$ and $D(n)^y$ and those by $D(n')^x$ and $D(n')^y$ for $n \neq n'$ are distinct. 

\begin{figure}[H]
\begin{center}
\includegraphics[height=5.0cm]{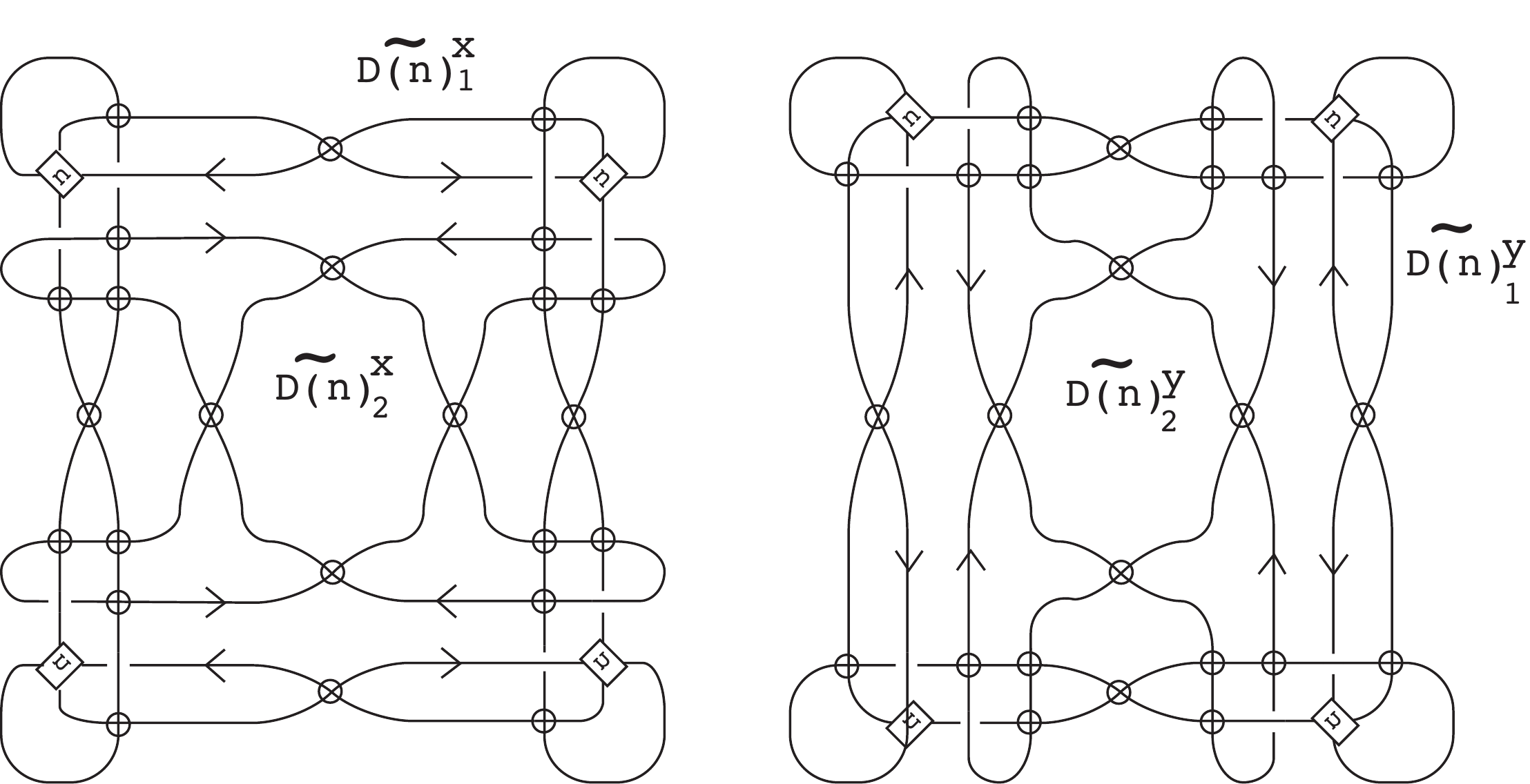}
\caption{Double covering diagrams $\widetilde{D}(n)^x$ and $\widetilde{D}(n)^y$}
\label{fig:sr-39G}
\end{center}
\end{figure}

Double cover diagrams 
$\widetilde{D}(n)^x$ and $\widetilde{D}(n)^y$ of $D^x(n)$ and $D^y(n)$ are depicted in Figure~\ref{fig:sr-39G}.  
Rotating $\widetilde{D}(n)^x$ by $90$ degrees, we obtain $\widetilde{D}(n)^y$.  
Thus they are equivalent as (partitioned) virtual links.   
\end{proof} 

\begin{remark}
We cannot use the index polynomial invariant $\psi$ for a proof of Theorem~\ref{thm:doubleB}. Let 
$D=D_1\cup\cdots\cup D_m$ and $D'=D'_1\cup\cdots\cup D'_m$ be twisted link diagrams such that their double covering diagrams are equivalent as partitioned virtual links.  
Then $D_i$ is of even type if and only if $D'_i$ is of even type.  In this case, we have that $d_i(D) = d_i(D')$ and 
$\psi^{(D_i,C_i)}(t) = \psi^{(D_i,C'_i)}(t)$ for any twisted intersection colorings $C_i$ and $C'_i$. This fact will be discussed elsewhere in a forthcoming paper. 
\end{remark}

\section{Construction of a pair of twisted links with equivalent double coverings}
\label{sect:construction}

In this section we give a method of constructing a pair of twisted link diagrams $D^x$ and $D^y$ whose double covering diagrams are equivalent as partitioned virtual links.  

Let $A$, $A_x$ and $A_y$ be subsets of $\R^2$ with 
$A=\{(x,y) \in \R^2 \mid 1 \leq x, 1 \leq y \}$, 
$A_x=\{(x,y) \in \R^2 \mid 1 \leq x, 0 \leq y \leq 1 \}$ 
and 
$A_y=\{(x,y) \in \R^2 \mid 0 \leq x \leq 1, 1 \leq y \}$.  

Let 
$D=D_1\cup\ldots\cup D_m$  be a twisted link diagram with $m$ components satisfying the following condition $(\ast)$:    

\begin{itemize}
\item[$(\ast)$] 
For each $i \in \{1, \dots, m\}$, there are exactly $2$ bars on $D_i$. One of the two bars is in $A_x$ and the other bar is in $A_y$. $D_i \cap A_x$ and $D_i \cap A_y$ are as in Figure~\ref{fig:sr-27}, where 
$D_1 \cap A_x, \dots, D_m \cap A_x$ and $D_1 \cap A_y, \dots, D_m \cap A_y$ are arranged in the order shown in the figure, and $D$ is contained in $A \cup A_x \cup A_y$.  
\end{itemize}

\begin{figure}[H]
\begin{center}
\includegraphics[height=3.5cm]{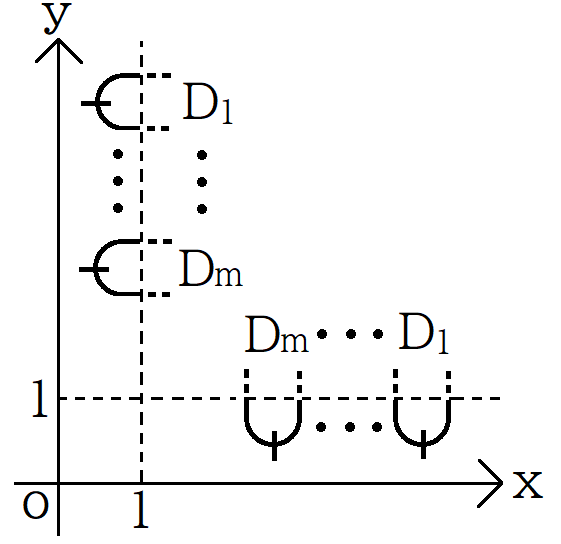}
\caption{The condition $(\ast)$ for $D=D_1\cup\ldots\cup D_m$}
\label{fig:sr-27}
\end{center}
\end{figure}

Let $s_x(D)=s_x(D_1)\cup\ldots\cup s_x(D_m)$ (or $s_y(D)=s_y(D_1)\cup\ldots\cup s_y(D_m)$) 
be a diagram obtained from $D=D_1\cup\ldots\cup D_m$ by applying a reflection along the $x$-axis (or $y$-axis) and then switching over/under information at all classical crossings.   
The union $D\cup s_x(D)$ (or $D\cup s_y(D)$) is depicted in Figure~\ref{fig:sr-28-29}. 

\begin{figure}[H]
\begin{tabular}{cc}
\begin{minipage}{0.40\hsize}
\begin{center}
\includegraphics[height=6.0cm]{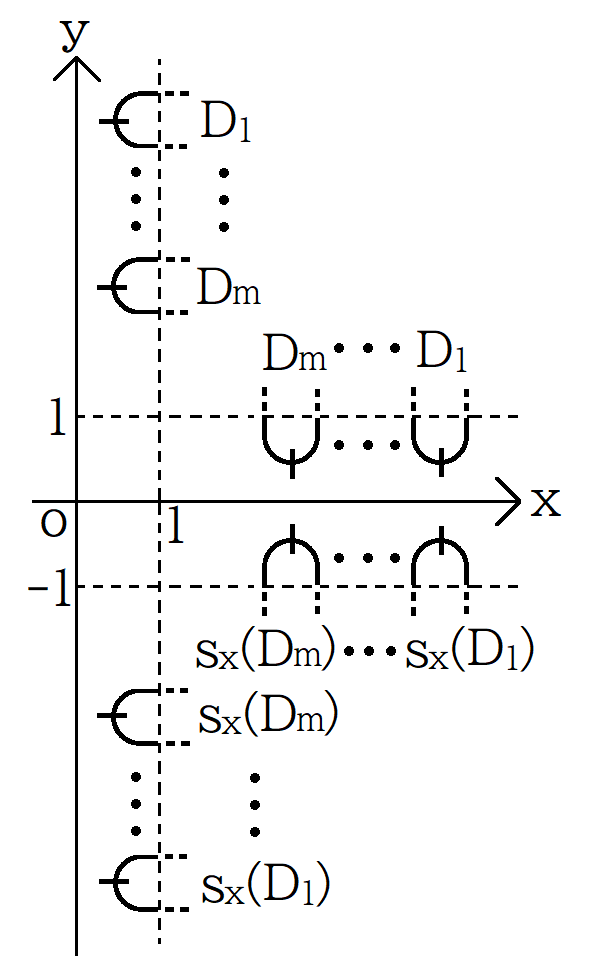}
$D\cup s_x(D)$
\end{center}
\end{minipage}

\begin{minipage}{0.60\hsize}
\begin{center}
\includegraphics[height=3.0cm]{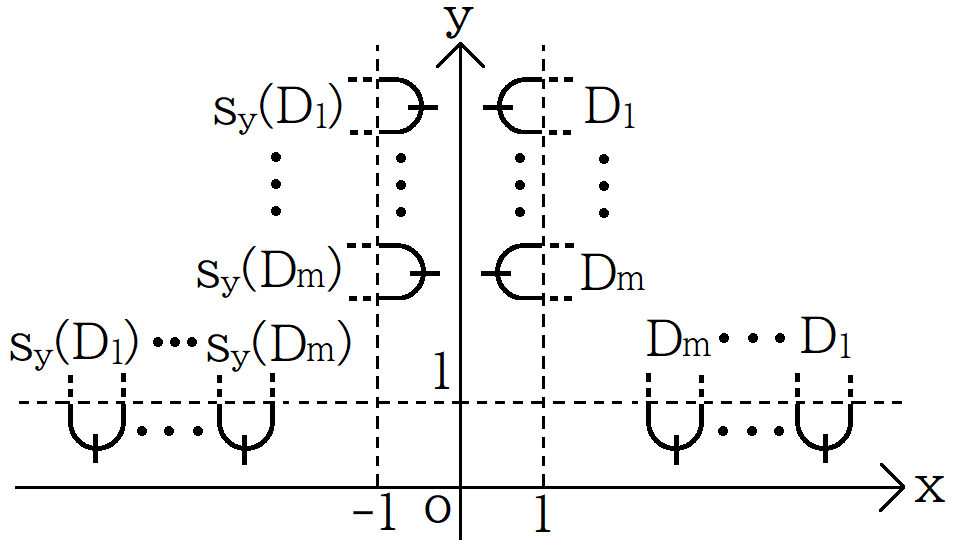}

$D\cup s_y(D)$
\end{center}
\end{minipage}

\end{tabular}
\caption{$D\cup s_x(D)$ and $D\cup s_y(D)$}
\label{fig:sr-28-29}
\end{figure}

Let $D^x=D_1^x\cup\ldots\cup D_m^x$ (or $D^y=D_1^y\cup\ldots\cup D_m^y$) be a diagram obtained from 
$D\cup s_x(D)$ (or $D\cup s_y(D)$) by replacing the arcs with bars with arcs with virtual crossings as depicted in Figure~\ref{fig:sr-30} (or Figure~\ref{fig:sr-31})

\begin{figure}[H]
\begin{center}
\includegraphics[height=2.3cm]{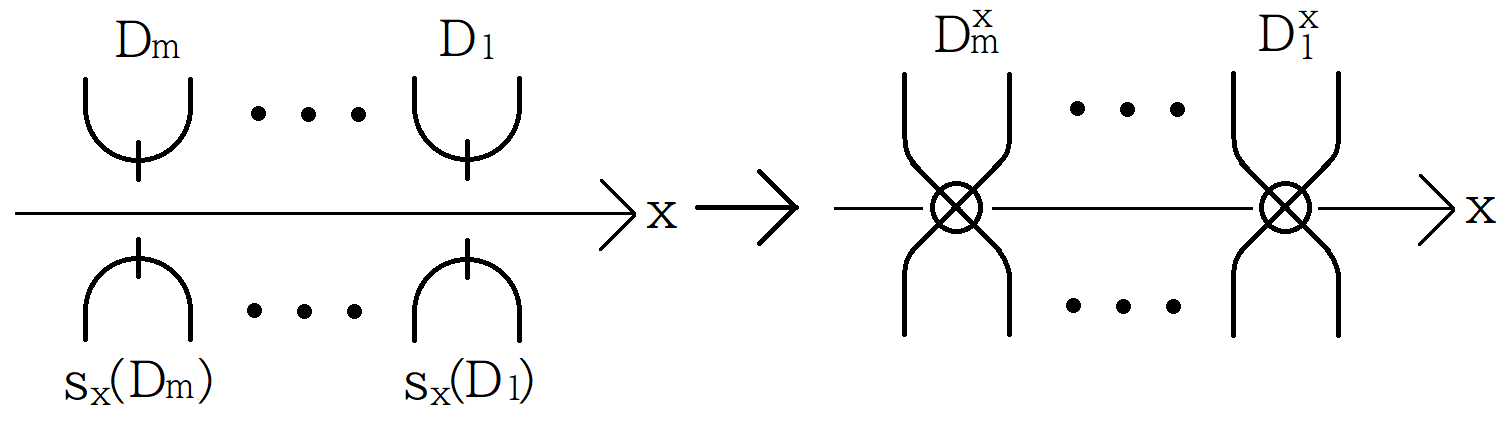}
\end{center}
\caption{Replacement to obtain $D^x=D_1^x\cup\ldots\cup D_m^x$}
\label{fig:sr-30}
\end{figure}

\begin{figure}[H]
\begin{center}
\includegraphics[height=3.5cm]{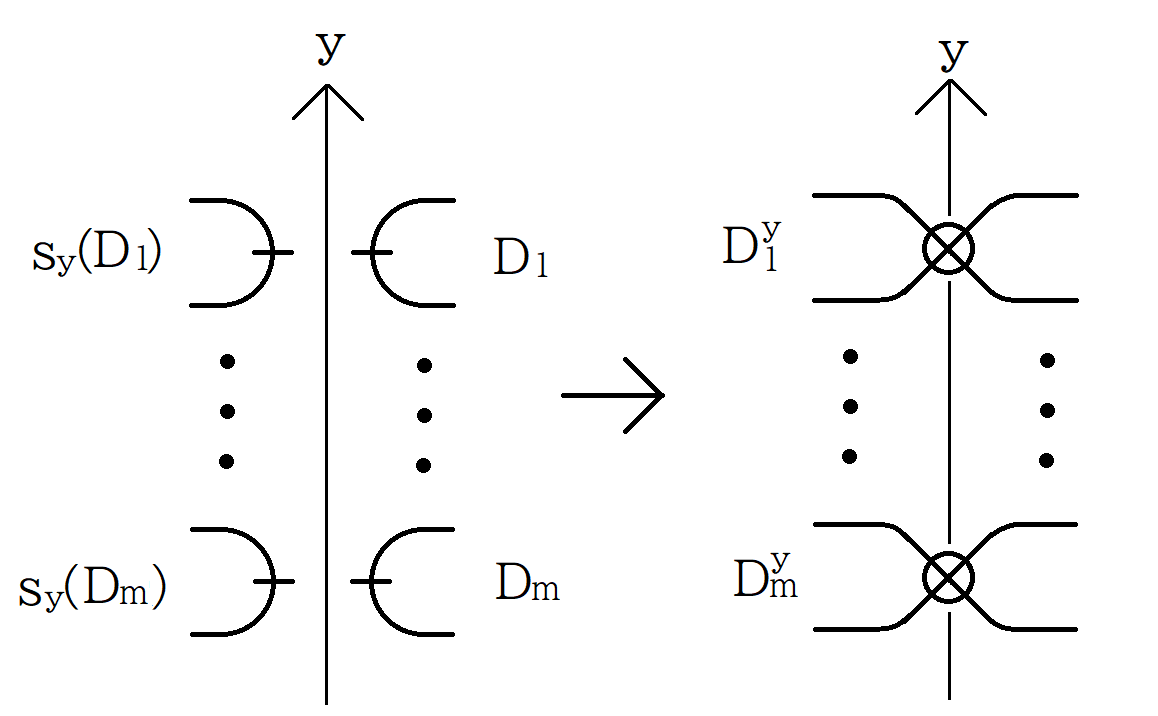}
\end{center}
\caption{Replacement to obtain $D^y=D_1^y\cup\ldots\cup D_m^y$}
\label{fig:sr-31}
\end{figure}

Then $D^x$ and $D^y$ are twisted virtual link diagrams with $m$ components such that all components are of even type, and     
their double covering diagrams $\widetilde{D}^x$ and $\widetilde{D}^y$ are considered as partitioned virtual link diagrams with $\{1, \dots, m\}$.  

\begin{proposition}
\label{prop:a}
In the situation above,  double covering diagrams $\widetilde{D}^x$ and $\widetilde{D}^y$ are equivalent as partitioned virtual links. 
\end{proposition}

\begin{proof}

Let $s_y\circ s_x(D)$ be the link diagram $s_y(s_x(D))$, which is the same with $s_x(s_y(D))$.  
Consider a diagram obtained from $D\cup s_x(D)\cup s_y(D)\cup s_y\circ s_x(D)$ by the replacement as in 
Figure~\ref{fig:sr-32}. Then it is $\widetilde{D}^x$, and also $\widetilde{D}^y$.  \end{proof}

\begin{figure}[H]
\begin{center}
\includegraphics[height=5.0cm]{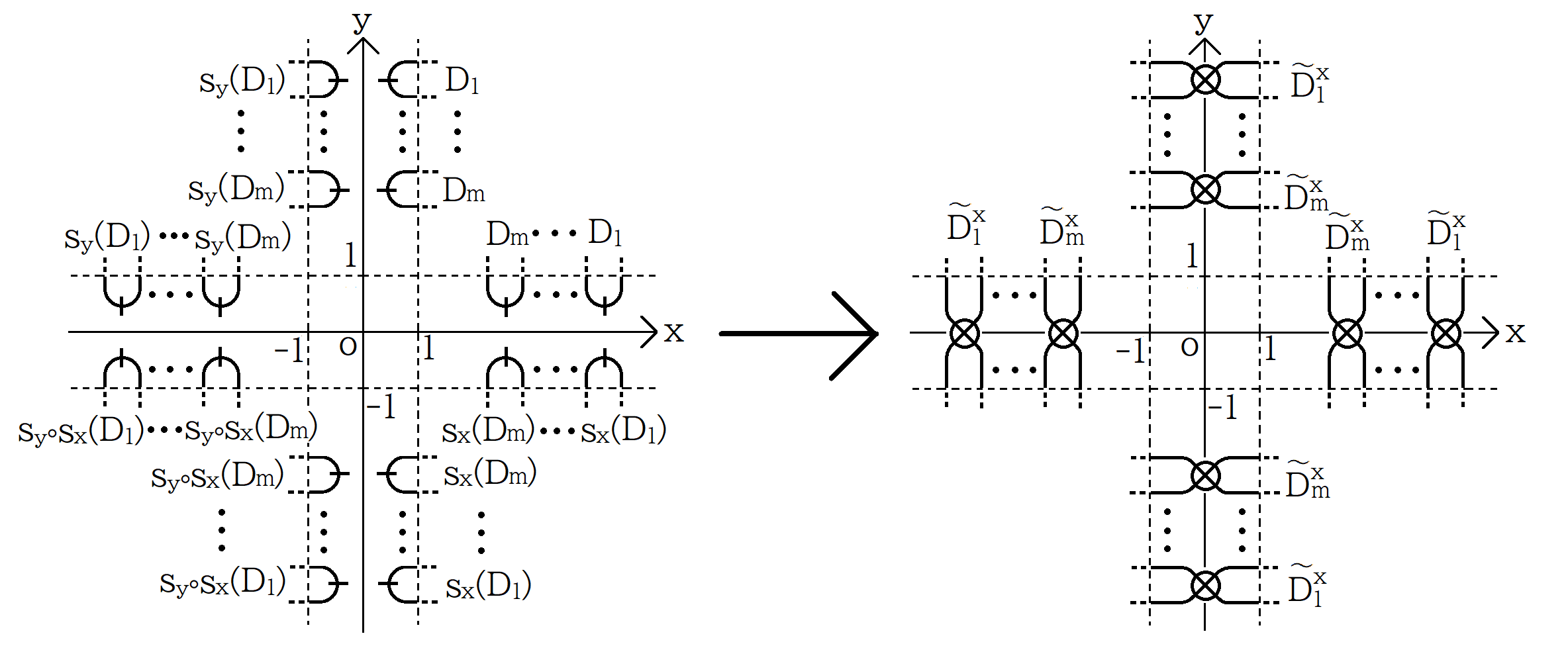}
\caption{$D\cup s_x(D)\cup s_y(D)\cup s_y\circ s_x(D)\longrightarrow\widetilde{D}^x=\widetilde{D}_1^x\cup\ldots\cup\widetilde{D}_m^x$}
\label{fig:sr-32}
\end{center}
\end{figure}

Let $D(n)=D(n)_1\cup D(n)_2$ be a diagram depicted in Figure~\ref{fig:sr-38B}.  Then 
$D(n)^x=D(n)_1^x\cup D(n)_2^x$ and $D(n)^y=D(n)_1^y\cup D(n)_2^y$ are as in Figure~\ref{fig:sr-39C}, where $D(n)^x$ is rotated by $90$ degrees.

\begin{figure}[H]
\begin{center}
\includegraphics[height=3.0cm]{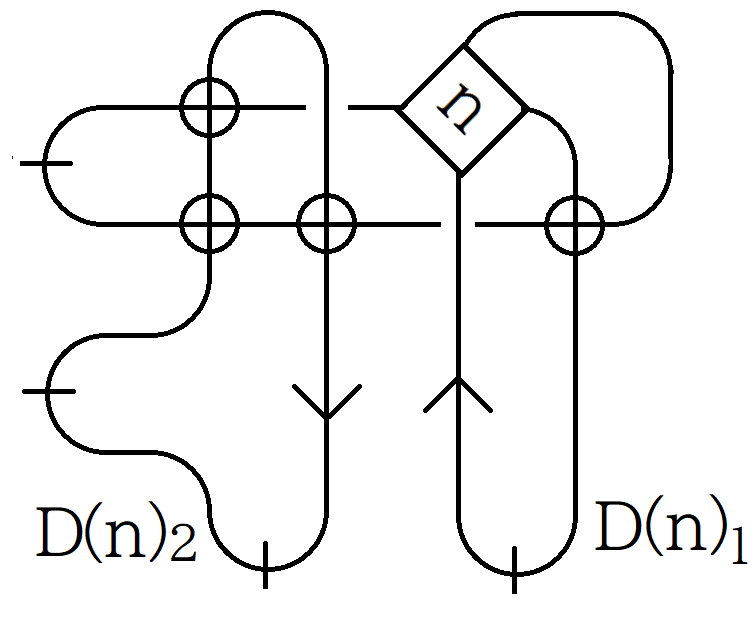}
\caption{$D(N)=D(n)_1\cup D(n)_2$}
\label{fig:sr-38B}
\end{center}
\end{figure}

%\addcontentsline{toc}{section}{References} 

%\bibliography{bibtex} 
%\bibliographystyle{junsrt} 

\end{document}